\documentclass[a4paper, 11pt, final]{article}

\usepackage[all]{xy}
\usepackage{style}
 \usepackage{amsthm, eufrak}
\usepackage{dsfont}

 \usepackage{geometry}

 \title{Projectors on the intermediate algebraic Jacobians}

\author{Charles Vial}
 \date{}

\begin{document}

\maketitle

\begin{abstract}
  Let $X$ be a complex smooth projective variety of dimension $d$.
  Under some assumption on the cohomology of $X$, we construct
  mutually orthogonal idempotents in $CH_d(X \times X) \otimes \Q$
  whose action on algebraically trivial cycles coincides with the
  Abel-Jacobi map. Such a construction generalizes Murre's
  construction of the Albanese and Picard idempotents and makes it
  possible to give new examples of varieties admitting a self-dual
  Chow-K\"unneth decomposition satisfying the motivic Lefschetz
  conjecture as well as new examples of varieties having a Kimura
  finite dimensional Chow motive. For instance, we prove that
  fourfolds with Chow group of zero-cycles supported on a curve (e.g.
  rationally connected fourfolds) have a self-dual Chow-K\"unneth
  decomposition which satisfies the motivic Lefschetz conjecture and
  consequently Grothendieck's standard conjectures. We also prove that
  hypersurfaces of very low degree are Kimura finite dimensional.
\end{abstract}

\section*{Introduction}

Let $X$ be a smooth projective variety of dimension $d$ over an
algebraically closed field $k\subset \C$. The Chow group $CH_i(X)$ of
cycles of dimension $i$ on $X$ is the $\Q$-vectorspace generated by
$i$-cycles on $X$ modulo rational equivalence. Given $\sim$ an
equivalence relation on cycles, $CH_i(X)_\sim$ denotes those cycles
which are $\sim 0$. In this paper, $\sim$ will either be algebraic,
homological or numerical equivalence. All three equivalence relations
agree on zero-cycles and are spanned by the zero cycles of degree
zero.

Being able to exhibit cycles in $CH_d(X \times X)$ with appropriate
action on the homology of $X$ is essential to Grothendieck's theory of
pure motives. As discussed for instance in \cite{Jannsen}, being able
to exhibit cycles in $CH_d(X \times X)$ which are idempotents is a
prerequisite to the understanding of Chow groups as part of the
framework of the Bloch-Beilinson-Murre philosophy. Roughly speaking,
such a framework predicts that the Chow groups of $X$ should be
controlled by the cohomology of $X$. In this paper we address a
question of a slightly different nature as whether the Chow groups of
$X$ dictate its Chow motive. Of course, we do not answer such a
question in generality. However, we completely answer this question in
the case when the Chow groups of $X$ are generated by the Chow groups
of zero-cycles of curves. For this purpose we construct appropriate
idempotents in $CH_d(X \times X)$.  In the spirit of the BBM
philosophy, work of Esnault and Levine \cite{EL} (and Jannsen
\cite{Jannsen} in the case of points) shows that if the Chow groups of
$X$ are generated by the Chow groups of curves, then the cohomology of
$X$ is generated by the cohomology of curves. Here, as a consequence
of the construction of appropriate idempotents, we show that if the
Chow groups of $X$ are generated by the Chow groups of curves, then
not only is the cohomology of $X$ generated by the cohomology of
curves but the Chow motive of $X$ is generated by the Chow motives of
curves (see theorem \ref{theorem2} below). In particular, this
complements Esnault and Levine's theorem by showing that the Chow
motive of $X$ is finite dimensional in the sense of Kimura
\cite{Kimura}.  The basic properties of pure motives are exposed in
\cite{Scholl} and the (covariant) notations we use are those of
\cite{KMP}.\medskip

Murre \cite{Murre} constructed mutually orthogonal idempotents $\Pi_1$
and $\Pi_{2d-1}$ in $CH_d(X \times X)$ called respectively the
Albanese projector and the Picard projector. Such idempotents satisfy
the following properties.
\begin{itemize}
\item $\Pi_1 = {}^t\Pi_{2d-1}$.
\item $(\Pi_1)_*H_*(X) = H_1(X)$ and $(\Pi_{2d-1})_*H_*(X) =
  H_{2d-1}(X)$.
\item $(\Pi_1)_*CH_*(X) = (\Pi_1)_*CH_0(X) \simeq \mathrm{Alb}_X(k)
  \otimes \Q$.
\item $(\Pi_{2d-1})_*CH_*(X) = CH_{d-1}(X)_{\hom} \simeq
  \mathrm{Pic}^0_X(k) \otimes \Q$.
\end{itemize}

Scholl \cite{Scholl} then showed that it is possible to modify
slightly the construction of these idempotents in order to, in
addition, have a Lefschetz isomorphism :
\begin{itemize}
\item The map $\iota_*\iota^* : (X,\Pi_{2d-1},0) \r (X,\Pi_{1},d-1)$
  is an isomorphism of Chow motives. Here $\iota : C \r X$ is a smooth
  linear section of dimension one of $X$.
\end{itemize}

In this paper, we wish to generalize Murre's construction in the
following sense : we wish to construct mutually orthogonal idempotents
$\Pi_{2i+1,i}$ in $CH_d(X \times X)$ which, in homology, define
projectors on the largest subHodge structure of $H_{2i+1}(X)$
generated by the $H_1$'s of curves. Here $H_k(X) := H_k(X(\C),\Q)$
which is isomorphic to $H^{2d-k}(X(\C),\Q)$. We offer two different
constructions. \medskip

The first construction is exposed in the first section. It is defined
for all smooth projective varieties $X$ but we cannot show that the
idempotents constructed there act appropriately in homology without
making some assumptions on $X$. In some sense the idempotents
constructed there lift the largest submotive of a curve contained in
the numerical motive of $X$. What is needed is Jannsen's
semi-simplicity theorem \cite{Jannsen3} in order to produce
idempotents modulo numerical equivalence, and then a lifting property
from numerical equivalence up to rational equivalence (proposition
\ref{prop}). \medskip

The second construction, which is much more precise, gives the
required idempotents but depends on an assumption on the cohomology of
$X$ which we describe below.  Let's define ${N}^i H_{2i}(X)$ to be the
image of the rational cycle class map $cl_i : CH_i(X) \r H_{2i}(X)$
and $$ {N}^i H_{2i+1}(X) := \sum \im \big( \Gamma_* : H_{1}(C) \r
H_{2i+1}(X) \big),$$ where the sum runs through all smooth projective
curves $C$ and through all correspondences $\Gamma \in CH_{i+1}( C
\times X)$. The use of the notation $N^iH_{2i+1}(X)$ is not arbitrary
since it can be shown that this subgroup of $H_{2i+1}(X)$ is spanned
by those classes that vanish in the open complement of some subvariety
of $X$ of dimension $i+1$.  The group $N^iH_{2i+1}(X)$ is thus the
last step of the coniveau filtration on $H_{2i+1}(X)$.

Given an integer $i$, the assumption we need on $X$ in order to
construct the idempotent that we will denote $\Pi_{i,\lfloor i/2
  \rfloor}$ is that the cup product pairing $H_{2d-i}(X) \times
H_{i}(X) \r \Q$ restricts to a non degenerate pairing on ${N}^{\lfloor
  (2d-i)/2\rfloor }H_{2d-i}(X) \times {N}^{\lfloor i/2\rfloor
}H_{i}(X)$. We begin the second section by showing in lemma
\ref{pairing-proj} that such pairings are non degenerate for a large
class of varieties. Lemma \ref{pairing-proj} also shows that these
pairings are expected to be non degenerate for all smooth projective
varieties if one believes in Grothendieck's standard conjectures.

The construction of the projectors $\Pi_{2i,i}$ is unsurprising and is
usually used to extract the N\'eron-Severi group $NS_i(X)$ out of
$CH_i(X)$ :

\begin{theorem2} \label{projector1}
   Let $i$ be an integer. Assume that the pairing
  ${N}^{d-i}H_{2d-2i}(X) \times {N}^{i}H_{2i}(X) \r \Q$ is non
  degenerate. Then there exist idempotents $\Pi_{2i,i}$ and
  $\Pi_{2d-2i,d-i}$ in $CH_d(X \times X)$ such that
  \begin{itemize}
  \item $\Pi_{2i,i} = {}^t\Pi_{2d-2i,d-i}$.
  \item $(\Pi_{2i,i})_*H_*(X) = {N}^{i}H_{2i}(X)$.
  \item $CH_i(X)_{\hom} = \ker \big(\Pi_{2i,i} : CH_i(X) \r
    CH_i(X)\big).$
  \item The Chow motive $(X,\Pi_{2i,i},0)$ is isomorphic to
    $(\L^{\otimes i})^{\oplus d_i}$ where $d_i = \dim N^iH_{2i}(X)$.
  \item If $2i \geq d$ there is a Lefschetz isomorphism of Chow
    motives $(X,\Pi_{2i,i},0) \r (X,\Pi_{2d-2i,d-i},2i-d) $ given by
    intersecting $2i-d$ times with a smooth hyperplane section of $X$.
  \end{itemize}
\end{theorem2}

We now turn to the construction of the projectors $\Pi_{2i+1,i}$. In
particular, our construction gives a motivic interpretation of the
Abel-Jacobi map to the algebraic part of the intermediate Jacobians.
Write $J_i^a(X)$ for the image of the Abel-Jacobi map $AJ_i :
CH_i(X)^\Z_{\alg} \r J_i(X(\C))$, it is an algebraic torus defined
over $k$.

\begin{theorem2} \label{projector2}
  Let $i$ be an integer. Assume that the pairing
  ${N}^{d-i-1}H_{2d-2i-1}(X) \times {N}^{i}H_{2i+1}(X) \r \Q$ is non
  degenerate. Then there exist idempotents $\Pi_{2i+1,i}$ and
  $\Pi_{2d-2i-1,d-i-1}$ in $CH_d(X \times X)$ such that
  \begin{itemize}
\item $\Pi_{2i+1,i} = {}^t\Pi_{2d-2i-1,d-i-1}$.
\item $(\Pi_{2i+1,i})_*H_*(X) = {N}^{i}H_{2i+1}(X)$.
\item $\ker \big( AJ_i : CH_i(X)_{\alg} \r J^a_i(X)(k) \otimes \Q
  \big) = \ker \big(\Pi_{2i+1,i} : CH_i(X)_{\alg} \r CH_i(X)_{\alg}
  \big).$
\item The Chow motive $(X,\Pi_{2i+1,i},0)$ is isomorphic to
  $\h_1(J_i^a(X))(i)$.
\item If $2i+1 \geq d$ there is a Lefschetz isomorphism of Chow
  motives $(X,\Pi_{2i+1,i},0) \r (X,\Pi_{2d-2i-1,d-i-1},2i+1-d) $
  given by intersecting $2i+1-d$ times with a smooth hyperplane
  section of $X$.
\end{itemize}
\end{theorem2}
These generalize Murre's construction of the Albanese and Picard
projectors ($\Pi_{1,0}$ and $\Pi_{2d-1,d-1}$ respectively) because in
the cases $i=0$ or $i=d-1$ we have ${N}^{0}H_{1}(X) =H_{1}(X)$ and
${N}^{d-1}H_{2d-1}(X)= H_{2d-1}(X)$. The pairing
${N}^{d-i-1}H_{2d-2i-1}(X) \times {N}^{i}H_{2i+1}(X) \r \Q$ is thus
just the cup product pairing between $H_{2d-1}(X)$ and $H_1(X)$ and is
always non degenerate.

Finally lemma \ref{pairing-proj} shows that the above pairings are all
non degenerate for curves, surfaces, abelian varieties, complete
intersections, uniruled threefolds, rationally connected fourfolds and
any smooth hyperplane section, product and finite quotient thereof.
For those varieties $X$ for which those idempotents can be constructed
for all $i$ we can show, thanks to the non-commutative Gram-Schmidt
process of lemma \ref{linalg}, that it is possible to choose the
idempotents of theorems \ref{projector1} and \ref{projector2} to be
pairwise orthogonal :

\begin{theorem2} \label{projector3} If the pairings ${N}^{\lfloor (2d-
    i)/2 \rfloor}H_{2d-i}(X) \times {N}^{\lfloor i/2 \rfloor }H_{i}(X)
  \r \Q$ are non degenerate for all $i$ then the idempotents of
  theorems \ref{projector1} and \ref{projector2} can be chosen to be
  pairwise orthogonal.
\end{theorem2}

The second section is then devoted to the proof of these theorems.
\medskip

In the third section, we compute the Chow motive of those varieties
whose Chow groups are all representable. We say that $CH_i(X)_\alg$ is
{\it representable} if there exists a curve $C$ and a correspondence
$\Gamma \in CH_{i+1}(C \times X)$ such that $CH_i(X)_\alg =
\Gamma_*CH_0(C)_\alg$. We show that if $X$ is a variety whose Chow
groups are all representable then the pairings ${N}^{\lfloor (2d- i)/2
  \rfloor}H_{2d-i}(X) \times {N}^{\lfloor i/2 \rfloor }H_{i}(X) \r \Q$
are non degenerate for all $i$.  We then use theorems
\ref{projector1}, \ref{projector2} and \ref{projector3} to compute the
Chow motive of varieties having representable Chow groups. Informally,
Esnault and Levine \cite{EL} showed that if the Chow groups of a
variety $X$ are all representable then the cohomology of $X$ is
generated by the cohomology of curves. Conversely Kimura \cite{Kimura}
proved that if the cohomology groups of $X$ are generated by the
cohomology of curves and if the Chow motive of $X$ is finite
dimensional (See \cite{Kimura} for a definition) then the Chow groups
of $X$ are representable.
Here we prove a stronger statement :

\begin{theorem2} \label{theorem2} Let $k \subseteq \C$ be an
  algebraically closed field. Let $X$ be a smooth projective variety
  of dimension $d$ over $k$. Write $X_\C:=X \times_{\mathrm{Spec} k}
  \mathrm{Spec} \ \C$ and $b_j:=\dim H_{j}(X)$. The following
  statements are equivalent.
  \begin{enumerate}
  \item $\h(X)= \mathds{1} \oplus \h_1(\mathrm{Alb}_X) \oplus
    \L^{\oplus b_{2}} \oplus \h_1(J^a_1(X))(1) \oplus (\L^{\otimes
      2})^{\oplus b_{4}} \oplus \cdots \oplus \h_1(J^a_{d-1}(X))(d-1)
    \oplus \L^{\otimes d}$.
  \item The cycle class maps $cl_i : CH_i(X) \r H_{2i}(X)$ and the
    rational Deligne cycle class maps $cl_i^{\mathcal{D}} : CH_i(X_\C)
    \r H^{\mathcal{D}}_{2i}(X,\Q(i))$ are surjective for all $i$ and
    $\h(X)$ is finite dimensional.
  \item The rational Deligne cycle class maps $cl_i^{\mathcal{D}} :
    CH_i(X_\C) \r H^{\mathcal{D}}_{2i}(X,\Q(i))$ are injective for all
    $i$.
   \item The Chow groups $CH_i(X_\C)_{\alg}$ are representable for all
    $i$.
  \end{enumerate}
\end{theorem2}

Esnault and Levine \cite{EL} proved that if the total Deligne cycle
class map of $X$ is injective then it is surjective. Theorem
\ref{theorem2} gives thus a better insight to the link between
injectivity and surjectivity of cycle class maps.

The proof of the theorem goes as follows. The first statement is
certainly the strongest, i.e. it implies the three others. The
equivalence $(3 \Leftrightarrow 4)$ is certainly known but we couldn't
find a reference for ($4 \Rightarrow 3$), so we include a proof in \S
3.2.  The proof relies on a generalized decomposition of the diagonal
and was essentially written in \cite{EL}. The implication $(2
\Rightarrow 4)$ is due to Kimura and appears in \cite[Theorem
7.10]{Kimura}.  Our main input is then a proof of $(4 \Rightarrow 1)$
which settles the theorem and which we give in \S 3.3. For sake of
completeness we also give a direct proof of $(2 \Rightarrow 1)$ in
\S3.1 using our projectors $\Pi_{2i,i}$ and $\Pi_{2i+1,i}$.

As an immediate corollary we get a generalization of a result by
Jannsen \cite[Th. 3.6]{Jannsen} who proved that if the total cycle
class map of $X$ is injective then it is surjective. This was
also proved by Kimura \cite{Kimura2}.

\begin{theorem2} \label{theorem1} Let $k \subseteq \C$ be an
  algebraically closed field.  Let $X$ be a smooth projective variety
  of dimension $d$ over $k$. The following statements are equivalent.
\begin{enumerate}
\item $\h(X) = \bigoplus_{i=0}^d (\L^{\otimes i})^{\oplus b_{2i}}$.
\item The rational cycle class map $cl : CH_*(X_\C) \r H_{*}(X)$ is
  surjective and $\h(X)$ is finite dimensional.
\item The rational cycle class maps $cl : CH_*(X_\C) \r H_{*}(X)$ is
  injective.
\item The Chow groups $CH_i(X_\C)$ are finite dimensional $\Q$-vector
  spaces for all $i$.
\end{enumerate}
\end{theorem2}
Again this theorem makes more precise the link between injectivity and
surjectivity of cycle class maps. \medskip

In the fourth and last section we are interested in using our
construction of idempotents to give new examples of varieties for
which we can compute explicitly a Chow-K\"unneth decomposition of the
diagonal. Such examples include $3$-folds $X$ satisfying
$H^2(X,\Omega_X)=0$ (e.g. Calabi-Yau $3$-folds), rationally connected
$4$-folds, and $4$-folds admitting a rational map to a curve with
rationally connected general fiber. We also prove that such varieties
satisfy the motivic Lefschetz conjecture and hence Grothendieck's
standard conjectures.

We are also interested in giving new examples of varieties whose Chow
motives are finite dimensional in the sense of Kimura \cite{Kimura}.
These will be given by smooth hyperplane sections of hypersurfaces
covered by a family of linear projective varieties of dimension
$\lfloor n/2 \rfloor $ in $\P^{n+1}$. These were considered by
Esnault, Levine and Viehweg \cite{ELV} and also subsequently by
Otwinowska \cite{Otwinowska} and include hypersurfaces of very small
degree, e.g.  cubic $5$-folds, $5$-folds which are the smooth
intersection of a cubic and a quadric and $7$-folds which are the
smooth intersection of two quadrics.  Other examples are given by
rationally connected threefolds, a case which was treated by
Gorchinskiy and Guletskii in \cite{GG}. \medskip


Let's mention that the construction given in the first section is used
in \cite{Vial1} to prove a generalization of the implication $(4
\Rightarrow 1)$ in theorem \ref{theorem2} to the case of Chow motives
with representable Chow groups. The proof given there does not involve
any cohomology theory, except implicitly through the use of Jannsen's
semi-simplicity theorem whose proof requires the existence of a
``good'' cohomology theory. In \cite{Vial2}, we prove Murre's
conjectures for the varieties considered in \S4. We also refer to
\cite[\S 2]{Vial2} for some statements that do not involve the
cohomology (or the Chow groups) of $X$ in all degrees.

\paragraph{Aknowledgements.} Thanks to Burt Totaro for useful comments. This work is supported by a Nevile
Research Fellowship at Magdalene College, Cambridge and an EPSRC
Postdoctoral Fellowship under grant EP/H028870/1. I would like to
thank both institutions for their support.

\section{A first construction} \label{projnum}

\paragraph{The coniveau filtration on numerical motives.} Let $k$ be
any field and $X$ a smooth projective variety over $k$ of dimension
$d$.  We refer to \cite[\S 7.1]{KMP} for the definition of pure
motives in a covariant setting. Chow motives are pure motives with
rational coefficients for rational equivalence and numerical motives
are pure motives with rational coefficients for numerical equivalence.
The category of Chow motives over $k$ is denoted $\M$ and the category
of numerical motives over $k$ is denoted $\bar{\M}$. The reduction
modulo numerical equivalence of a cycle $\gamma \in CH_k(X)$ is
denoted $\bar{\gamma}$. A fundamental result of Jannsen
\cite{Jannsen3} states that the category of numerical motives is
abelian semi-simple. In particular if $f: N \r M$ is a morphism of
numerical motives with $M = (X,p,n)$, Jannsen's theorem gives the
existence of a correspondence $\pi \in \End_k(M)$ such that $\im f
\simeq (X,\pi,n)$.

It is thus possible to define a coniveau filtration on numerical
motives as in \cite[\S 8]{An} and \cite[\S 7.7]{KMP} : $N^jM := \sum
\big(f: \bar{\h}(Y)(j) \r M \big)$ where the sum runs through all
smooth projective varieties $Y$ and all morphisms $f \in \Hom_k(
\bar{\h}(Y)(j), M )$ and where $\bar{\h}(Y)(j)$ denotes the numerical
motive of $Y$ tensored $j$ times by the Lefschetz motive.

Let's imagine for a moment that Grothendieck's standard conjecture B
(cf. \cite[5.2.4.1]{An}) is true. Then \cite[5.4.2.1]{An} each
numerical motive $M$ has a weight decomposition that we write $M =
\bigoplus_iM_i$. Furthermore, for weight reasons $N^jM_i = \sum
\big(f: \bar{\h}_{i-2j}(Y)(j) \r M_i \big)$. Another consequence of
Grothendieck's standard conjecture B is that if $i :Z \r Y$ is a
smooth hyperplane section of dimension $i-2j$ of $Y$ then $i_* :
\bar{\h}_{i-2j}(Z) \r \bar{\h}_{i-2j}(Y)$ is surjective
\cite[5.2.5.1]{An}. Therefore we have ${N}^j M_i = \sum \im \big(f :
\bar{\h}_{i-2j}(Y)(j) \r M_i \big),$ where the sum runs through all
smooth projective varieties $Y$ with $\dim Y = i-2j$ and through all
morphisms $f \in \Hom_k(\bar{\h}_{i-2j}(Y)(j), M_i)$.

\paragraph{The idempotents $\bar{\pi}_{2j,j}$ and $\bar{\pi}_{2j+1,j}$. }
Let's now forget about the standard conjectures. We know that points
and curves have a weight decomposition \cite[4.3.2]{An}; it is
therefore natural for any integer $j$ and for any numerical motive $M$
to consider the following direct summands of $M$ : $$ M_{2j,j} := \sum
\im \big(f : \bar{\h}_{0}(\Spec k)(j) \r M \big) \ \ \ \mathrm{and} \
\ \ M_{2j+1,j} := \sum \im \big(f : \bar{\h}_{1}(C)(j) \r M \big)$$
where the first sum runs through all morphisms $f \in
\Hom_k(\bar{\h}_0(\Spec k),M)$ and the second sum runs through all
curves $C$ and through all morphisms $f \in \Hom_k(\bar{\h}_1(C),M)$.
Thus in particular there exist for all integers $j$ corespondences
${\pi}_{2j,j}$ and ${\pi}_{2j+1,j}$ in $CH_d(X \times X)$ such that
$M_{2j,j} = (X,\bar{{\pi}}_{2j,j},0)$ and $M_{2j+1,j} =
(X,\bar{{\pi}}_{2j+1,j},0)$.

\paragraph{A lifting property.}

We denote by $\M_0$ (resp. $\M_1$) the full thick subcategory of $\M$
generated by the Chow motives of points (resp. the $\h_1$'s of smooth
projective curves over $k$).
For a motive $P \in \M$, let $\bar{P}$ denote its image in $\bar{\M}$.
(This notation is abusive since we previously denoted numerical
motives with a bar and it is not known if all numerical motives admit
a lift to rational equivalence).  Let's also write $\bar{\M}_0$ (resp.
$\bar{\M}_1$) for the image of $\M_0$ (resp. $\M_1$) in $\bar{\M}$.
The functors $\M_0 \r \bar{\M}_0$ and $\M_1 \r \bar{\M}_1$ are
equivalence of categories (see \cite[corollary 3.4]{Scholl}) and as
such the categories $\M_0$ and $\M_1$ are abelian semi-simple.

\begin{proposition} \label{prop} Let $M$ be an object in $\M_0$ (resp.
  in $\M_1$). Let $N$ be any motive in $\M$. Then any morphism $f : M
  \r N$ induces a splitting $N = N_1 \oplus N_2$ with $N_1$ isomorphic
  to an object in $\M_0$ (resp. in $\M_1$) and $\bar{N}_1 \simeq \im
  \bar{f}$.
\end{proposition}
\begin{proof}
  The morphism $M \r N$ induces a morphism $\bar{M} \r \bar{N}$ and it
  is known that any morphism in an abelian semi-simple category is a
  direct sum of a zero morphism and of an isomorphism (cf.
  \cite[A.2.13]{AnKa}). Let's thus write $$ \bar{M} = \bar{M}_1 \oplus
  \bar{M}_2 \stackrel{\bar{f} \oplus 0}{\longrightarrow} \bar{N}_1
  \oplus \bar{N}_2 = \bar{N}$$ where $\bar{f}$ is an isomorphism
  $\bar{M}_1 \stackrel{\sim}{\longrightarrow} \bar{N}_1$. The
  composition $(\bar{f}^{-1} \oplus 0) \circ (\bar{f} \oplus 0) \in
  \End(\bar{M})$ is therefore equal to the projector $\bar{M}\r
  \bar{M}_1 \r \bar{M}$. Let then $g : N \r M$ be any lift of
  $\bar{f}^{-1} \oplus 0 : \bar{N} \r \bar{M}$ and let $M_1$ be any
  lift of $\bar{M}_1$. Then $g \circ f \in \End(M)$. But it is a fact
  that $\End(M) =\End(\bar{M})$. Therefore, $g \circ f$ is a projector
  on $M_1$. We now claim that $f \circ g \circ f \circ g$ defines a
  projector in $\End(N)$ onto an object isomorphic to $M_1$. Indeed,
  $$(f \circ g \circ f \circ g) \circ (f \circ g \circ f \circ g) = f
  \circ \big( (g \circ f) \circ (g \circ f) \circ (g \circ f) \big)
  \circ g = f \circ ( g \circ f) \circ g = f \circ g \circ f \circ g$$
  and we have the commutative diagram
  \begin{center} $ \xymatrix{ N \ar[r]^g & M \ar[r]^f \ar[dr] & N
      \ar[r]^g & M \ar[r]^f \ar[dr] & N \ar[r]^g & M \ar[r]^f
      \ar[dr] &   N \ar[r]^g &  M \ar[r]^f &   N  \\
      & & M_1 \ar[rr]^{\id} \ar[ur] & & M_1 \ar[rr]^{\id} \ar[ur] & &
      M_1 \ar[ur] & &}$
\end{center} showing that indeed $f \circ g \circ f \circ g$ projects
onto $M_1$ (since it has a retraction).
\end{proof}

\paragraph{The idempotents ${\pi}_{2j,j}$ and ${\pi}_{2j+1,j}$. }
Proposition \ref{prop} shows that it is actually possible to choose
the correspondences ${\pi}_{2j,j}$ and ${\pi}_{2j+1,j}$ above to be
idempotents in $CH_d(X \times X)$. In other words, proposition
\ref{prop} shows that it is possible to define direct summands
$(X,{\pi}_{2j,j},0)$ and $(X,{\pi}_{2j+1,j},0)$ of the Chow motive
$\h(X)$ of $X$ whose reduction modulo numerical equivalence are the
direct summands $M_{2j,j}$ and $M_{2j+1,j}$ defined above.  \medskip

We won't be giving the details here but it can be shown that, if
Grothendieck's standard conjecture B is true for all smooth projective
varieties, then the idempotents ${\pi}_{2j,j}$ and ${\pi}_{2j+1,j}$
constructed here coincide modulo homological equivalence with the
idempotents ${\Pi}_{2j,j}$ and ${\Pi}_{2j+1,j}$ of \S2.

\paragraph{A remark about the K\"unneth projectors.}

We would like to explain how it is possible to construct cycles whose
numerical classes are the expected K\"unneth projectors, i.e. whose
homological classes are expected to be the projectors $H_*(X) \r
H_i(X) \r H_*(X)$. We proceed by induction on $d = \dim X$. If $X =
\mathrm{Spec} \ k$, we define $\pi_0^X$ to be the cycle $X \times X$
inside $X \times X$. Suppose we have constructed projectors modulo
numerical equivalence $\pi_0^Y, \pi_1^Y, \ldots, \pi_{2(\dim Y)-2}^Y$
for all smooth projective varieties $Y$ of dimension $\dim Y < d$.
Then, for all $i \in \{0,\ldots, d-1\}$, we define the cycle $\pi_i^X
\in CH_d(X\times X)/\num$ to be the projector such that
$$(X,\pi_i^X,0) = \bigcup_{f : Y \r X} \im \big(f_* : (Y,\pi^Y_i,0) \r
\bar{\h}(X)\big),$$ where the sum runs through all smooth projective
varieties $Y$ of dimension $i$ and all morphisms $f : Y \r X$.  We
then set $\pi_{2d-i}^X = {}^t\pi_i^X$ and $\pi_d^X = \id_X -
\sum_{i\neq d} \pi_i^X$.

If Grothendieck's standard conjecture B is true, then it can be
checked that those define the expected K\"unneth projectors.

\section{The projectors $\Pi_{2i,i}$ and $\Pi_{2i+1,i}$}

In this section, we fix an algebraically closed field $k$ with an
embedding $k \hookrightarrow \C$ and we prove theorems
\ref{projector1}, \ref{projector2} and \ref{projector3}. We start with
a lemma which shows that many varieties do satisfy the assumptions of
these theorems. \medskip

Let $X$ be a $d$-dimensional smooth projective variety over $k$.  Let
$\iota : H \r X$ be a smooth hyperplane section of $X$ and let
$\Gamma_\iota \in CH_{d-1}(H \times X)$ be its graph and let
${}^t\Gamma_\iota$ be the transpose of $\Gamma_\iota$. We define $L:=
\Gamma_\iota \circ {}^t\Gamma_\iota \in CH_{d-1}(X \times X)$. The
hard Lefschetz theorem asserts that the map $L^i : H_{d+i}(X) \r
H_{d-i}(X)$ given by intersecting $i$ times with $H$ is an isomorphism
for all $i \geq 0$. The variety $X$ is said to satisfy property B if
the inverse morphism is induced by an algebraic correrspondence for
all $i \geq 0$. It is one of Grothendieck's standard conjectures that
all smooth projective varieties should satisfy B.  \medskip

\begin{lemma} \label{pairing-proj} Let $i$ be an integer in
  $\{d,\ldots,2d\}$. The cup product pairing ${N}^{\lfloor (2d- i)/2
    \rfloor}H_{2d-i}(X) \times {N}^{\lfloor i/2 \rfloor }H_{i}(X) \r
  \Q$ is non degenerate in either of the following cases:
  \begin{itemize}
  \item $X$ satisfies property B.
  \item ${N}^{\lfloor i/2 \rfloor }H_{i}(X) = H_i(X)$.
  \end{itemize}
  In particular the pairing ${N}^{d-1}H_{2d-1}(X) \times {N}^{0
  }H_{1}(X) \r \Q$ is non degenerate for all $X$.
\end{lemma}
\begin{proof} In the case $X$ satisfies property B, the Hodge index
  theorem is a crucial ingredient and the lemma is a special case of
  \cite[Prop.  1.4]{Vial2}. The other case is obvious.
\end{proof}

Therefore the results in this section hold for curves, surfaces,
abelian varieties, complete intersections, uniruled threefolds,
rationally connected fourfolds and any smooth hypersurface section,
product or finite quotient thereof. All those varieties satisfy
property B. Lemma \ref{essential} shows that the results in this
section also hold for varieties for which some cycle class map is
surjective.

\subsection{Setup} \label{setup}

We are given a smooth projective variety $X$ over $k$ of dimension
$d$. By definition $ N^{i}H_{2i}(X)$ coincides with the image of the
cycle class map $CH_{2i}(X) \r H_{2i}(X)$.  For each integer $i$, let
$d_i = \dim_\Q N^iH_{2i}(X)$ and let $P_i$ be the disjoint union of
$d_i$ copies of $\Spec \ k$. Notice that $d_i>0$ because $N^iH_{2i}(X)$
always contains the $(d-i)$-fold intersection of a hyperplane section.
We then fix $\Gamma_{2i} \in CH_i(P_i \times X)$ such that
$$(\Gamma_{2i})_* : H_0(P_i) \stackrel{\simeq}{\longrightarrow}
N^{i}H_{2i}(X)$$ is bijective. This amounts to fixing a basis of
$N^iH_{2i}(X) = \im (cl_i : CH_i(X) \r H_{2i}(X))$.\medskip

For each integer $i$, we also fix a smooth projective curve (not
necessarily connected) $C_i$ and a correspondence $\Gamma_{2i+1} \in
CH_{i+1}(C_i \times X)$ such that $$(\Gamma_{2i+1})_*H_1(C_i) =
N^iH_{2i+1}(X).$$ Let $C_{i,l}$ be the connected components of $C_i$
and for all $l$ let $z_{i,l}$ be a rational point on $C_{i,l}$. Up to
composing $\Gamma_{2i+1}$ with the correspondence $\Delta_{C_i} -
\sum_l \big(\{z_{i,l}\}\times C_{i,l} + C_{i,l} \times \{z_{i,l}\}
\big) \in CH_1(C_i \times C_i)$, we can and we will assume moreover
that $$(\Gamma_{2i+1})_*H_0(C_i) = (\Gamma_{2i+1})_*H_2(C_i) = 0.$$
\medskip

In order to establish the Lefschetz isomorphism of theorems
\ref{projector1}, \ref{projector2} and \ref{projector3} we will make
use of the following easy lemma.

\begin{lemma} \label{lef} Let $i$ be an integer in $\{d+1,\ldots,2d\}$
  and assume that the cup product pairing ${N}^{\lfloor (2d- i)/2
    \rfloor}H_{2d-i}(X) \times {N}^{\lfloor i/2 \rfloor }H_{i}(X) \r
  \Q$ is non degenerate. Then $L^{i-d} : H_i(X) \r H_{2d-i}(X)$ maps
  isomorphically ${N}^{\lfloor i/2 \rfloor }H_{i}(X)$ to ${N}^{\lfloor
    (2d- i)/2 \rfloor}H_{2d-i}(X)$.
\end{lemma}

\begin{proof}
  The non degeneracy assumption says in particular that the two $\Q$
  vectorspaces ${N}^{\lfloor i/2 \rfloor }H_{i}(X)$ and ${N}^{\lfloor
    (2d- i)/2 \rfloor}H_{2d-i}(X)$ have same dimension.  The Lefschetz
  isomorphism $L$ restricts to an injective map ${N}^{\lfloor i/2
    \rfloor }H_{i}(X) \r H_{2d-i}(X)$ and, by definition of $N$, maps
  ${N}^{\lfloor i/2 \rfloor }H_{i}(X)$ into ${N}^{\lfloor (2d- i)/2
    \rfloor}H_{2d-i}(X)$.
\end{proof}

\begin{remark}
  In fact if the pairing ${N}^{\lfloor (2d- i)/2 \rfloor}H_{2d-i}(X)
  \times {N}^{\lfloor i/2 \rfloor }H_{i}(X) \r \Q$ is non degenerate,
  then more is true. Namely, as a consequence of the Lefschetz
  isomorphisms of propositions \ref{lef1} and \ref{lef2}, we have that
  the isomorphism $L^{i-d} : {N}^{\lfloor i/2 \rfloor }H_{i}(X)
  \r{N}^{\lfloor (2d- i)/2 \rfloor}H_{2d-i}(X)$ has its inverse
  induced by a correspondence.
\end{remark}

If for an integer $i$ such that $2i \in \{d, \ldots,2d\}$, the cup
product pairing ${N}^{d-i}H_{2d-2i}(X) \times {N}^{i}H_{2i}(X) \r \Q$
is non degenerate, lemma \ref{lef} makes it possible to furthermore
assume that $P_i = P_{d-i}$ and $\Gamma_{2d-2i} = L^{2i-d} \circ
\Gamma_{2i}$.

Likewise if for an integer $i$ such that $2i+1 \in \{d, \ldots,2d\}$,
the cup product pairing ${N}^{d-i-1}H_{2d-2i-1}(X) \times
{N}^{i}H_{2i+1}(X) \r \Q$ is non degenerate, lemma \ref{lef} makes it
possible to furthermore assume that $C_i = C_{d-i-1}$ and
$\Gamma_{2d-2i-1} = L^{2i+1-d} \circ \Gamma_{2i+1}$.

\subsection{Proof of theorem \ref{projector1}}

In this paragraph we consider an integer $i$ with $2i \geq d$ and a
smooth projective variety $X$ of dimension $d$ for which the pairing
$N^{d-i}H_{2d-2i}(X) \times N^{i}H_{2i}(X) \r \Q$ is non degenerate.
\medskip

The correspondence $\Gamma_{2d-2i} := L^{2i-d} \circ \Gamma_{2i}$
induces by duality a bijective map
$$({}^t\Gamma_{2d-2i})_* : \big( N^{d-i}H_{2d-2i}(X)\big)^\vee
\stackrel{\simeq}{\longrightarrow} H_0(P_{i})^\vee.$$ By the non
degeneracy assumption $\big( N^{d-i}H_{2d-2i}(X)\big)^\vee$ identifies
with $ N^{i}H_{2i}(X)$ and $H_0(P_{i})^\vee$ identifies with
$H_0(P_{i})$. Therefore the composition ${}^t\Gamma_{2d-2i} \circ
\Gamma_{2i}$ induces a $\Q$-linear isomorphism $H_0(P_i)
\stackrel{\simeq}{\longrightarrow} H_0(P_{i})$.  It is then clear that
there exists a correspondence $\gamma_{2i} \in CH_0(P_{i} \times
P_{i})$ such that $\gamma_{2i} \circ {}^t\Gamma_{2d-2i} \circ
\Gamma_{2i}$ acts as identity on $H_0(P_i)$.  Because $L = {}^tL$ we
also have $$\gamma_{2i} = {}^t\gamma_{2i} .$$

We then set $$\Pi_{2i,i} := \Gamma_{2i} \circ \gamma_{2i} \circ
{}^t\Gamma_{2d-2i} \in CH_d(X\times X).$$ Since $\gamma_{2i} \circ
{}^t\Gamma_{2d-2i} \circ \Gamma_{2i} = \id \in CH_0(P_i \times P_i)$,
it is clear that $\Pi_{2i,i}$ is an idempotent and that it induces the
projector $H_*(X) \r N^iH_{2i}(X) \r H_*(X)$ in homology. Also, it is
clear that $\Pi_{2d-2i,d-i} := {}^t\Pi_{2i,i}$ (Notice that if $2i=d$
then $\Pi_{d,d/2} = {}^t\Pi_{d,d/2}$) defines an idempotent which
induces the projector $H_*(X) \r N^{d-i}H_{2d-2i}(X) \r H_*(X)$ in
homology. \qed

\begin{proposition} \label{lef1} The correspondence $\Pi_{2d-2i,d-i} \circ
  L^{2i-d} \circ \Pi_{2i,i} : (X,\Pi_{2i,i},0) \r
  (X,\Pi_{2d-2i,d-i},2i-d)$ is an isomorphism of Chow motives.
\end{proposition}
\begin{proof} Using the identities $L = {}^tL$, $\gamma_{2i} =
  {}^t\gamma_{2i} $, $\Pi_{2d-2i,d-i} = {}^t\Pi_{2i,i}$ and the fact
  that $\Pi_{2i,i} = \Gamma_{2i} \circ \gamma_{2i} \circ
  {}^t\Gamma_{2i} \circ {}^tL^{2i-d}$ is an idempotent, one can easily
  check that $\Pi_{2i,i} \circ \Gamma_{2i} \circ \gamma_{2i} \circ
  {}^t\Gamma_{2i} \circ \Pi_{2d-2i,d-i}$ is the inverse of
  $\Pi_{2d-2i,d-i} \circ L^{2i-d} \circ \Pi_{2i,i}$.
\end{proof}

\begin{proposition} \label{Chow1} Let $\Pi_{2i} \in CH_d(X \times X)$
  be an idempotent which factors through a zero-dimensional variety
  $P_i$ as $\Pi_{2i} = \Gamma \circ \alpha$ with $\Gamma \in
  CH_{i}(P_i \times X)$ and $\alpha \in CH_{d-i}(X \times P_i)$, and whose
  action on $H_*(X)$ is the orthogonal projection on $N^iH_{2i}(X)$.
  Then the Chow motive $(X,\Pi_{2i},0)$ is isomorphic to
  $(\mathds{L}^{\otimes i})^{\oplus d_i}$.
\end{proposition}

\begin{proof} The cycle class map $CH_0(P_i) \r H_0(P_i)$ is an
  isomorphism. Let $\pi := \alpha \circ \Gamma \in CH_0(P_i \times
  P_i)$. By functoriality of the cycle class map, we see that $\pi$ is
  an idempotent such that $(P_i,\pi,0) = \mathds{1}^{\oplus d_i}$.
  The correspondence $\Gamma$ is an element of $ CH_i(P_i \times X) =
  \Hom_{k}\big(\h(P_i)(i), \h( X)\big)$ and it can easily be checked
  that the correspondence $\Pi_{2i}\circ \Gamma \circ \pi \in
  \Hom_{k}\big((P_i,\pi,i), (X,\Pi_{2i},0)\big)$ is an isomorphism
  with inverse $\pi \circ \alpha \circ \Pi_{2i}$.
\end{proof}

\begin{proposition}\label{hom}  Let $Q_{2i}$ be a correspondence
  in $CH_d(X \times X)$ such that $Q_{2i}$ acts as the identity on
  $N^iH_{2i}(X)$ and such that $Q_{2i}$ is supported on $X \times Z$
  with $Z$ a subvariety of $X$ of dimension $i$. Then $CH_i(X)_{\hom}
  = \ker \big(Q_{2i} : CH_i(X) \r CH_i(X) \big).$ In particular
  $$CH_i(X)_{\hom} = \ker \big(\Pi_{2i,i} : CH_i(X) \r CH_i(X)\big).$$
\end{proposition}

\begin{proof}
  By functoriality of the cycle class map, we have a commutative
  diagram \begin{center} $ \xymatrix{ CH_i(X) \ar[d]^{cl_i} \ar[r] &
      CH^0(Z) \ar[d]^{\simeq} \ar[r] & CH_i(X) \ar[d]^{cl_i} \\
      H_{2i}(X) \ar[r] & H^0(Z) \ar[r] & H_{2i}(X)}$ \end{center}
  The composition of the two arrows of the top row is the map induced
  by $Q_{2i}$ and the composition of the two arrows of the bottom
  row is the identity on $\im(cl_i)$.  The proposition follows easily.
\end{proof}

\subsection{Proof of theorem \ref{projector2}}

\paragraph{The Albanese and the Picard varieties.}

Let $X$ be a smooth projective variety over a field $k$. The Albanese
variety attached to $X$ and denoted $\mathrm{Alb}_X$ is an abelian
variety universal for maps $X \r A$ from $X$ to abelian varieties $A$
sending a fixed point $x_0 \in X$ to $0 \in A$. The Picard variety
$\mathrm{Pic}^0_X$ of $X$ is the abelian variety parametrizing
numerically trivial line bundles on $X$ (i.e. those with vanishing
Chern class). These define respectively a covariant and a
contravariant functor from the category of smooth projective varieties
to the category of abelian varieties.

The abelian varieties $\mathrm{Alb}_X$ and $\mathrm{Pic}^0_X$ are dual
and are isogenous in the following way. Let $C$ be a curve which is a
smooth linear section of $X$. Then the map $$ \Psi : \mathrm{Pic}^0_X
\r \mathrm{Pic}^0_C \stackrel{\Theta}{\longrightarrow} \mathrm{Alb}_C
\r \mathrm{Alb}_X$$ is an isogeny, where $\Theta$ is the map induced
by the theta-divisor on the curve $C$.\medskip

\noindent The following proposition is essential to the construction
of the idempotents $\Pi_{2i+1,i}$.

\begin{proposition}[\textit{cf.} th. 3.9 and prop. 3.10 of
  \cite{Scholl}] \label{Scholl} Let $Y$ and $Z$ be connected smooth
  projective varieties and let $\zeta \in CH_0(Y)$ and $\eta \in
  CH_0(Z)$ be $0$-cycles of positive degree.  Then, there is an
  isomorphism
  $$\Omega : \Hom(\mathrm{Alb}_Y, \mathrm{Pic}^0_Z) \otimes \Q
  \stackrel{\simeq}{\longrightarrow} \{ c \in CH^1(Y \times Z) \ /
  \ c(\zeta)={}^tc(\eta)=0 \}.$$ Moreover, $\Omega$ is functorial in
  the following sense. Let $\phi : Y' \r Y$ and $\psi : Z' \r Z$ be
  morphisms of varieties and let $\zeta'$ and $\eta'$ be positive
  $0$-cycles on $Y'$ and $Z'$ with direct image $\zeta$ and $\eta$ on
  $Y$ and $Z$. If $\beta : \mathrm{Alb}_Y \r \mathrm{Pic}^0_Z$ is a
  homomorphism, then $$\Omega ( \mathrm{Pic}_\psi \circ \beta ) =
  \psi^* \circ \Omega(\beta) \ \ \mathrm{and} \ \ \Omega(\beta \circ
  \mathrm{Alb}_\phi) = \Omega(\beta) \circ \phi_*,$$ where $\Omega$ is
  taken with respect to the chosen $0$-cycles.
\end{proposition}

\paragraph{Intermediate Jacobians.}

Given a smooth projective complex variety $X$, the $i^\mathrm{th}$
intermediate Jacobian attached to $X$ is the compact complex torus
$$J_i(X) = \frac{H_{2i+1}(X,\C)}{F^i H_{2i+1}(X,\C) +
  H_{2i+1}(X,\Z)}.$$ It comes with a map $$AJ_i : CH^\Z_i(X)_{\hom} \r
J_i(X)$$ defined on the integral Chow group called the $i^\mathrm{th}$
Abel-Jacobi map which was thoroughly studied by Griffiths
\cite{Griffiths}. In the cases $i=0$ and $i=\dim X - 1$, we recover
the notions of Albanese variety and Picard variety respectively. These
intermediate Jacobians are however fairly different since they are of
transcendental nature.  While the Albanese and the Picard variety are
algebraic tori, this is not the case in general for intermediate
Jacobians. Precisely, let $J_i^\mathrm{alg}$ denote the maximal
subtorus inside $J_i(X)$ whose tangent space is included in
$H_{i+1,i}(X,\C)$. It is then a fact that $J_i^\mathrm{alg}$ is an
abelian variety and that
$$J_i^\mathrm{alg}(X) = \frac{N_H^i H_{2i+1}(X,\C)}{N_H^i
  H_{2i+1}(X,\C) \cap \big(F^i H_{2i+1}(X,\C) + H_{2i+1}(X,\Z)
  \big)}$$ where $N_H^i H_{2i+1}(X)$ is the maximal sub-Hodge
structure of $H_{2i+1}(X)$ contained in $H_{i+1,i}(X,\C) \oplus
H_{i,i+1}(X,\C)$. In particular, the intermediate Jacobian is
algebraic if and only if $H_{2i+1}(X,\C)$ is concentrated in degrees
$(i,i+1)$ and $(i+1,i)$. As a consequence of the horizontality of
normal functions associated to algebraic cycles \cite{Griffiths}, the
cycles in $CH^\Z_i(X)_{\hom}$ that are algebraically trivial map into
$J_i^\mathrm{alg}(X)$ under the Abel-Jacobi map. The map
$CH^\Z_i(X)_{\alg} \r J_i^\mathrm{alg}(X)$ is surjective if
$N^iH_{2i+1}(X) \supseteq N_H^iH_{2i+1}(X)$ (the reverse inclusion
always holds), in particular if $N^iH_{2i+1}(X) =H_{2i+1}(X)$. In any
case, let's write $J_i^a(X)$ for the image of the map $AJ_i :
CH^\Z_i(X)_\alg \r J_i^\alg(X)$. It is an abelian subvariety of the
abelian variety $J_i^\alg(X)$ which is defined the same way as
$J_i^\alg(X)$ with $N_H$ replaced with $N$.  We sum this up in the
commutative diagram
\begin{center} $ \xymatrix{
    CH^\Z_i(X)_{\hom} \ar[rr]^{\ \ \ AJ_i} & & J_i(X) \\
    CH^\Z_i(X)_{\alg} \ar@{->>}[r]^{\ \ \ AJ_i} \ar@{^(->}[u] & J_i^a(X)
    \ar@{^(->}[r] & J_i^\mathrm{alg}(X). \ar@{^(->}[u] }$ \end{center}

Finally if $X$ is defined over an algebraically closed subfield $k$ of
$\C$, the image of the composite map $$CH^\Z_i(X)_\alg \r
CH^\Z_i(X_\C)_\alg \r J_i^\alg(X_\C)$$ defines an abelian variety over
$k$ that we denote $J_i^a(X)$.

\paragraph{The projectors $\Pi_{2i+1,i}$ and  $\Pi_{2d-2i-1,d-i-1}$.}

Given any abelian varieties $A$ and $B$, $\Hom(A,B)$ denotes the group
of homomorphisms from $A$ to $B$.  Recall that the category of abelian
varieties up to isogeny is the category whose objects are the abelian
varieties and whose morphisms are given by $\Hom(A,B)
\otimes_\mathbf{Z} \Q$ for any abelian varieties $A$ and $B$. This
category is abelian semi-simple, cf. \cite{Weil}. \medskip

In the rest of this paragraph we consider an integer $i$ with $2i+1
\geq d$ and a smooth projective variety $X$ of dimension $d$ for which
the pairing $N^{d-i-1}H_{2d-2i-1}(X) \times N^{i}H_{2i+1}(X) \r \Q$ is
non degenerate. In particular the dual of $J_{d-i-1}^a(X)$ identifies
with $J_i^a(X)$. Lemma \ref{lef} implies that the correspondence
$L^{2i+1-d}$ induces an isogeny $\Lambda : J_i^a(X) \r J_{d-i-1}^a(X)$
and because $L = {}^tL$ we have $\Lambda = \Lambda^\vee$, i.e.
$\Lambda$ is equal to its dual. \medskip

Taking up what was said in \S \ref{setup} we have a smooth projective
curve $C_i$ over $k$ and correspondences $\Gamma_{2i+1} \in CH_{i+1} (
C_i \times X)$ and $\Gamma_{2d-2i-1} := L^{2i+1-d} \circ \Gamma_{2i+1}
\in CH_{d-i}(C_i \times X)$ such that both maps
$$(\Gamma_{2i+1})_* : H_{1}(C_i) \r  {N}^iH_{2i+1}(X) \ \ \mathrm{and}
\ \ (\Gamma_{2d-2i-1})_* : H_{1}(C_i) \r {N}^{d-i-1}H_{2d-2i-1}(X)$$
are surjective and such that both maps act trivially on $H_0(C_i)$ and
on $H_2(C_i)$.
The correspondence $\Gamma_{2i+1}$ induces by functoriality of the
Abel-Jacobi map a surjective homomorphism
$$(\Gamma_{2i+1})_* : \mathrm{Alb}_{C_i} \twoheadrightarrow J_i^a(X)$$
as well as a homomorphism with finite kernel
$$({}^t\Gamma_{2i+1})_* \circ \Lambda : J_i^a(X)
\hookrightarrow \mathrm{Pic}^0_{C_i}.$$
By semisimplicity of the category of abelian varieties up to isogeny,
there exists $\alpha \in \Hom(J_i^a(X),\mathrm{Alb}_{C_i})\otimes \Q$
such that $(\Gamma_{2i+1})_* \circ \alpha = \id_{J_i^a(X)}$. Let's
consider $$\Phi := \alpha \circ \Lambda^{-1} \circ \alpha^\vee \in
\Hom(\mathrm{Pic}^0_{C_i},\mathrm{Alb}_{C_i}) \otimes \Q$$ so that
$$(\star) \hspace{20pt} (\Gamma_{2i+1})_* \circ \Phi \circ
({}^t\Gamma_{2i+1})_* \circ \Lambda = \id_{J_i^a(X)}.$$

We would now like to use proposition \ref{Scholl} in order to give an
algebraic origin to $\Phi$. Decomposing $C_i$ into the disjoint union
of its connected components $C_{i,l}$, proposition \ref{Scholl} gives
a functorial isomorphism between $\Hom(\mathrm{Alb}_{C_i},
\mathrm{Pic}^0_{C_i}) \otimes \Q$ and $\{ c \in CH^1(C_i \times C_i) \
/ \ c(z_{i,l})={}^tc(z_{i,l})=0 \}$ for $z_{i,l}$ the rational point
on $C_{i,l}$ considered in \S \ref{setup}. Here, $\Phi$ belongs to
$\Hom(\mathrm{Pic}^0_{C_i}, \mathrm{Alb}_{C_i}) \otimes \Q$ which is
not quite the Hom group in the statement of proposition \ref{Scholl}.
To correct this, let $\Theta$ denote the theta-divisor of the curve
$C_i$.  Then under the isomorphism of proposition \ref{Scholl}, $\Phi$
corresponds to a correspondence $\gamma_{2i+1} := \Theta \circ \Gamma'
\circ \Theta^{-1} \in CH^1(C_i \times C_i)$ satisfying
$(\gamma_{2i+1})_* z_{i,l} = ({}^t\gamma_{2i+1})_* z_{i,l} = 0$ for
all $l$. Because $\Phi = \Phi^\vee$ we have $$\gamma_{2i+1} =
{}^t\gamma_{2i+1}.$$

We now set $$\Pi_{2i+1,i} := \Gamma_{2i+1} \circ \gamma_{2i+1} \circ
{}^t\Gamma_{2d-2i-1} = \Gamma_{2i+1} \circ \gamma_{2i+1} \circ
{}^t\Gamma_{2i+1} \circ L^{2i+1-d} \in CH_d(X \times X).$$ By
($\star$), the action of $\Pi_{2i+1,i}$ on $J_i^a(X)$ is given by
$\id_{J_i^a(X)}$.  The fact that $\Pi_{2i+1,i}$ defines a projector
goes as follows.  It is enough to prove that $$\gamma_{2i+1} \circ
{}^t\Gamma_{2i+1} \circ \Lambda \circ \Gamma_{2i+1} \circ
\gamma_{2i+1} = \gamma_{2i+1}.$$ Thanks to proposition \ref{Scholl},
it is actually enough to prove $\Phi \circ ({}^t\Gamma_{2i+1})_* \circ
L^{2i+1-d}_* \circ (\Gamma_{2i+1})_* \circ \Phi = \Phi :
\mathrm{Pic}^0_{C_i} \rightarrow \mathrm{Alb}_{C_i} $. This last
statement follows directly from ($\star$).

Now because $\Pi_{2i+1,i}$ acts as the identity on $J_i^a(X)$ and
because $\Gamma_{2i+1}$ and $\Gamma_{2d-2i-1}$ act trivially on
homology classes of degree $\neq 1$, we see that the homology class of
$\Pi_{2i+1,i}$ is the projector $H_*(X) \r N^i H_{2i+1}(X) \r H_*(X)$.

Finally we set $\Pi_{2d-2i-1,d-i-1} := {}^t \Pi_{2i+1,i}$, which is
licit in the case $2d-2i-1 = d$ since in this case $\gamma_{2i+1} =
{}^t\gamma_{2i+1}$ implies $ \Pi_{2i+1,i} = {}^t \Pi_{2i+1,i}$.  It is
then straightforward that $\Pi_{2d-2i-1,d-i-1}$ defines and idempotent
that induces the projector $H_*(X) \r N^{d-i-1} H_{2d-2i-1}(X) \r
H_*(X)$. \qed

\begin{proposition} \label{lef2} The correspondence
  $\Pi_{2d-2i-1,d-i-1} \circ L^{2i+1-d} \circ \Pi_{2i+1,i} :
  (X,\Pi_{2i+1,i},0) \r (X,\Pi_{2d-2i-1,d-i-1},2i+1-d)$ is an
  isomorphism of Chow motives.
\end{proposition}
\begin{proof}
  Using the identities $L = {}^tL$, $\gamma_{2i+1} =
  {}^t\gamma_{2i+1}$, $\Pi_{2d-2i-1,d-i-1} = {}^t\Pi_{2i+1,i}$ and the
  fact that $\Pi_{2i+1,i} = \Gamma_{2i+1} \circ \gamma_{2i+1} \circ
  {}^t\Gamma_{2i+1} \circ {}^tL^{2i+1-d}$ is an idempotent, one can
  easily check that $\Pi_{2i+1,i} \circ \Gamma_{2i+1} \circ
  \gamma_{2i+1} \circ {}^t\Gamma_{2i+1} \circ \Pi_{2d-2i-1,d-i-1}$ is
  the inverse of $\Pi_{2d-2i-1,d-i-1} \circ L^{2i+1-d} \circ
  \Pi_{2i+1,i}$.
\end{proof}

\begin{proposition} \label{Chow2} Let $\Pi_{2i+1} \in CH_d(X \times
  X)$ be an idempotent which factors through a curve $C_i$ as
  $\Pi_{2i+1} = \Gamma \circ \alpha$ with $\Gamma \in CH_{i+1}(C_i
  \times X)$ and $\alpha \in CH_{d-i}(X \times C_i)$, and whose action on
  $H_*(X)$ is the orthogonal projection on $N^iH_{2i+1}(X)$. Then the
  Chow motive $(X,\Pi_{2i+1,i},0)$ is isomorphic to
  $\h_1(J_i^a(X))(i)$.
\end{proposition}

\begin{proof}


  The assumption on the homology class of $\Pi_{2i+1}$ implies that
  $\Pi_{2i+1}$ acts as the identity on $J_i^a(X)$ and acts as zero on
  $H_{2i}(X)$. Therefore, by functoriality of the cycle class map,
  $\alpha \circ \Gamma \in CH^1(C_i \times C_i)$ acts as zero on some
  positive degree zero-cycle $\zeta$ on $C_i$.  Now a consequence of
  proposition \ref{Scholl} is that given two abelian varieties $J$ and
  $J'$ over $k$, there is a canonical identification
  $$\Hom(J,J')\otimes \Q = \Hom_{k}\big(\h_1(J),\h_1(J')\big).$$
  Because $\Pi_{2i+1}$ acts as the identity on $J_i^a(X)$, $\alpha
  \circ \Gamma$ defines an idempotent $\pi \in
  \End(\h_1(C_i))$ such that $(C_i,\pi,0) \simeq \h_1(J_i^a(X))$.

  \noindent The correspondence $\Gamma$ seen as a morphism of motives
  belongs to $\Hom_k \big( \h(C_i)(i),\h(X)\big)$. Let's show that
  $$\Pi_{2i+1,i} \circ \Gamma \circ \pi \in \Hom_k \big(
  (C_i,\pi,i),(X,\Pi_{2i+1,i},0) \big)$$ is an isomorphism. In fact,
  let's show that its inverse is given by $\pi \circ \alpha \circ
  \Pi_{2i+1,i}$, i.e that \begin{center} $(\Pi \circ \Gamma \circ \pi)
    \circ (\pi \circ \alpha \circ \Pi) = \Pi \ \ \mathrm{and} \ \ (\pi
    \circ \alpha \circ \Pi) \circ (\Pi \circ \Gamma \circ \pi) = \pi$
    as correspondences,
  \end{center} where for convenience we have dropped the subscripts
  ``$2i+1$''. But then, this is obvious because $\Pi=\Gamma \circ
  \alpha$ and $\pi = \alpha \circ \Gamma$ are idempotents.
\end{proof}

\begin{proposition} \label{jacobi} Let $Q_{2i+1}$ be a correspondence
  in $CH_d(X \times X)$ such that $Q_{2i+1}$ acts as the identity on
  $N^iH_{2i+1}(X)$ and such that $Q_{2i+1}$ is supported on $X \times
  Z$ with $Z$ a subvariety of $X$ of dimension $i+1$. Then $\ker \big(
  AJ_i : CH_i(X)_{\alg} \r J_i(X) \otimes \Q \big) = \ker
  \big(Q_{2i+1} : CH_i(X)_{\alg} \r CH_i(X)_{\alg} \big).$ In
  particular $$\ker \big( AJ_i : CH_i(X)_{\alg} \r J_i(X) \otimes \Q
  \big) = \ker \big(\Pi_{2i+1,1} : CH_i(X)_{\alg} \r CH_i(X)_{\alg}
  \big).$$
\end{proposition}

\begin{proof} The assumptions on $Q_{2i+1}$ imply that the action of
  $Q_{2i+1}$ on $CH_i(X)$ factors through $CH^1(\widetilde{Z})$ for
  some desingularization $\widetilde{Z} \r Z$; they also imply that
  the induced action of $Q_{2i+1}$ on $J_i^a(X)$ is the identity.

  We have thus the commutative diagram
  \begin{center} $ \xymatrix{CH_i(X)_{\alg} \ar[d]^{AJ_i}
      \ar[r]^{A} & CH^1(\widetilde{Z})_{\alg} \ar[d]^{\simeq} \ar[r]^{B} &
      CH_i(X)_{\alg} \ar[d]^{AJ_i} \\ J_i^a(X) \otimes
      \Q \ar[r] & \mathrm{Pic}_{\widetilde{Z}}^0 \otimes \Q \ar[r] &
      J_i^a(X)\otimes \Q}$
  \end{center}
  where $A$ and $B$ are correspondences such that $B \circ A =
  Q_{2i+1}$.  The inclusion $\ker AJ_i \subseteq \ker \Pi_{2i+1,i}$
  follows from the commutativity of the diagram, which itself is a
  consequence of the functoriality of the Abel-Jacobi map with respect
  to the action of correspondences.  The reverse inclusion $\ker
  \Pi_{2i+1,i} \subseteq \ker AJ_i$ follows from the fact that the
  composite of the two lower horizontal arrows is the identity on $\im
  AJ_i = J_i^a(X) \otimes \Q$.
\end{proof}

\begin{remark}
  An interesting question is to decide whether or not the action of an
  idempotent on homology determines its action on Chow groups. For
  example, given idempotents $\pi_{2i,i}$ and $\pi_{2i+1,i} \in CH_d(X
  \times X)$ such that $(\pi_{2i,i})_*H_*(X)=N^iH_{2i}(X)$ and
  $(\pi_{2i+1,i})_*H_*(X)=N^iH_{2i+1}(X)$, do we have
  \begin{center}$ CH_i(X)_{\hom} = \ker \big(\pi_{2i,i} : CH_i(X) \r
  CH_i(X)\big)$ \medskip

  and $\ker \big( AJ_i : CH_i(X)_{\alg} \r J_i^a(X)\otimes \Q \big) =
  \ker \big(\pi_{2i+1,i} : CH_i(X)_{\alg} \r CH_i(X)_{\alg} \big)?$
\end{center} It is shown in \cite{Vial2} that this is the case if $X$
is finite dimensional in the sense of Kimura.
\end{remark}

\subsection{Proof of theorem \ref{projector3}}

In this section we are given a smooth projective variety $X$ of
dimension $d$ for which the pairings are all non degenerate. As such,
by theorems \ref{projector1} and \ref{projector2} we can define all
the idempotents $\Pi_{2i,i}$ and $\Pi_{2i+1,i}$. However these are not
all necessarily pairwise orthogonal. We start with the following
linear algebra lemma which makes it possible to modify the idempotents
so as to make them pairwise orthogonal.

\begin{lemma} \label{linalg} Let $V$ be a $\Q$-algebra and let $k$ be
  a positive integer. Let $\pi_0, \ldots, \pi_n$ be idempotents in $V$
  such that $\pi_j \circ \pi_i = 0$ whenever $j -i < k$ and $j \neq
  i$. Then the endomorphisms $$p_i := (1-\frac{1}{2}\pi_n) \circ
  \cdots \circ (1-\frac{1}{2}\pi_{i+1}) \circ \pi_i \circ
  (1-\frac{1}{2}\pi_{i-1}) \circ \cdots \circ (1-\frac{1}{2}\pi_0)$$
  define idempotents such that $p_j \circ p_i = 0$ whenever $j -i <
  k+1$ and $j \neq i$.
\end{lemma}

\begin{proof}
  Let $j$ and $i$ be such that $j-i<k+1$ and look at $$\Pi := \pi_j
  \circ (1-\frac{1}{2}\pi_{j-1}) \circ \cdots \circ
  (1-\frac{1}{2}\pi_{0}) \circ (1-\frac{1}{2}\pi_n) \circ \cdots \circ
  (1-\frac{1}{2}\pi_{i+1}) \circ \pi_i.$$ Suppose first $j<i$. Because
  we have $\pi_r \circ \pi_s = 0$ for all $r<s$, we immediately see
  that $\Pi=0$. Suppose $j=i$, it is also easy to see that in this
  case $\Pi= \pi_i$. Finally, suppose that $i<j<i+k+1$. Because $\pi_r
  \circ \pi_s = 0$ for all $r < s+k$, we can see after expanding $\Pi$
  that $\Pi = \pi_j \circ \pi_i - \frac{1}{2}\pi_j \circ \pi_i \circ
  \pi_i - \frac{1}{2}\pi_j \circ \pi_j \circ \pi_i = 0$.
\end{proof}

In our case of concern, we get

\begin{theorem} \label{GS} Let $X$ be a smooth projective variety of
  dimension $d$.  Let $i<d$ be an integer and let $\pi_0, \ldots,
  \pi_i \in CH_d(X \times X)$ be idempotents such that
  $(\pi_j)_*H_*(X) = H_j(X)$ for all $0 \leq j \leq i$. Let
  $\pi_{2d-j} := {}^t\pi_j$ for $0 \leq j \leq i$. If $\pi_r \circ
  \pi_s =0$ for all $0 \leq r < s \leq 2d$, then the non-commutative
  Gram-Schmidt process of lemma \ref{linalg} gives mutually orthogonal
  idempotents $\{p_j\}_{j \in \{0,\ldots,i,2d-i,\ldots2d\}}$ such that
  $(p_j)_*H_*(X) = H_j(X)$ and $p_{2d-j} := {}^tp_j$ for all $j \in
  \{0,\ldots,i,2d-i,\ldots2d\}$. Moreover, we have isomorphisms of
  Chow motives $(X,\pi_j) \simeq (X,p_j)$ for all $j$.
\end{theorem}
\begin{proof}
  In order to get mutually orthogonal idempotents, it is enough to
  apply lemma \ref{linalg} $2i+2$ times. In order to prove the
  theorem, it suffices to prove each statement after each application
  of the process of lemma \ref{linalg}. Everything is then clear,
  except perhaps for the last statement. The isomorphism is simply
  given by the correspondence $p_j \circ \pi_j$ and its inverse is
  $\pi_j \circ p_j$.
\end{proof}

\begin{proposition}
  The projectors of theorems \ref{projector1} and \ref{projector2}
  satisfy
  \begin{itemize}
  \item $\Pi_{2i,i} \circ \Pi_{2j,j}=0$ for $i \neq j$.
  \item $\Pi_{2i+1,i} \circ \Pi_{2j+1,j}=0$ for $|i-j|>1$.
  \item $\Pi_{2i+1,i} \circ \Pi_{2j,j}=0$ for $|i-j|>1$.
  \item $\Pi_{2i,i} \circ \Pi_{2j+1,j}=0$ for $|i-j|>1$.
  \end{itemize}
\end{proposition}
\begin{proof}
Look at the dimension of ${}^t\Gamma_{2d-i} \circ \Gamma_j$.
\end{proof}

\begin{proposition}
  The projectors of theorems \ref{projector1} and \ref{projector2}
  satisfy
  \begin{itemize}
  \item $\Pi_{2i-1,i-1} \circ \Pi_{2i+1,i}=0$ for all $i$.
  \item $\Pi_{2i,i} \circ \Pi_{2i+1,i}=0$ and $\Pi_{2i+1,i} \circ
    \Pi_{2i+2,i+1}=0$ for all $i$.
  \end{itemize}
\end{proposition}
\begin{proof}
  For the first point we have ${}^t\Gamma_{2d-2i+1} \circ
  \Gamma_{2i+1} \circ \gamma_{2i+1} \in CH_2(C_i \times C_{i-1})$ and
  thus there exist rational numbers $a_{l,l'}$ such that
  ${}^t\Gamma_{2d-2i+1} \circ \Gamma_{2i+1} \circ \gamma_{2i+1} =
  \sum_{l,l'} a_{l,l'}[C_{i,l} \times C_{i-1,l'}]$. This yields
  $({}^t\Gamma_{2d-2i+1} \circ \Gamma_{2i+1} \circ
  \gamma_{2i+1})_*z_{i,l} = \sum_{l'} a_{l,l'}[C_{i-1,l'}]$. By
  definition of $\gamma_{2i+1}$ we also have $(\gamma_{2i+1})_*z_{i,l}
  = 0$ for all $l$. Hence $ a_{l,l'} = 0$ for all $l$ and all $l'$.
  Therefore ${}^t\Gamma_{2d-2i+1} \circ \Gamma_{2i+1} \circ
  \gamma_{2i+1} =0$.

  For the second point, up to transposing it is enough to prove one of
  the two equalities. Let's prove the second one. We have
  ${}^t\Gamma_{2d-2i-1} \circ \Gamma_{2i+2} \in CH_1(P_{i+1} \times
  C_{i})$. But then, because ${}^t\gamma_{2i+1}$ acts trivially on
  $z_{i.l} \in CH_0(C_{i,l})$ for all $l$, we see that $\gamma_{2i+1}$
  acts trivially on $CH_1(C_i)$. Therefore $\gamma_{2i+1} \circ
  {}^t\Gamma_{2d-2i-1} \circ \Gamma_{2i+2} =0$.
\end{proof}

\begin{remark} \label{remark} We have shown through the two previous
  propositions that ${}^t\Gamma_{2d-j} \circ \Gamma_i \circ \gamma_i =
  0$ for $j - i<0$ and in particular that $\Pi_{j,\lfloor j/2 \rfloor}
  \circ \Pi_{i,\lfloor i/2 \rfloor} =0$ for $j - i< 0$.
\end{remark}

\begin{remark} \label{nottrue} The missing orthogonal relations are
  $\Pi_{2i+1,i} \circ \Pi_{2i,i}=0$, $\Pi_{2i+2,i+1} \circ
  \Pi_{2i+1,i}=0$ or $ \Pi_{2i+1,i} \circ \Pi_{2i-1,i-1} =0$. There is
  no reason that these should hold true for the idempotents
  constructed in \S\S 2.2. and 2.3.
\end{remark}

Before we proceed to the proof of theorem \ref{projector3} we need a
lemma.

\begin{lemma} \label{leflemma} If $i \geq d$ then ${}^t\Pi_j \circ
  L^{i-d} \circ \Pi_{j'} = 0$ for $j,j' \geq i$ except in the case
  $i=j=j'$.
\end{lemma}
\begin{proof}
  Up to transposing we only have to prove ${}^t\Pi_j \circ L^{i-d}
  \circ \Pi_{j'} = 0$ for $j' \geq j \geq i$ not all equal. In fact it
  is enough to prove $\gamma_j \circ {}^t\Gamma_j \circ L^{i-d} \circ
  \Gamma_{j'} \circ \gamma_{j'} = 0$ for $j' \geq j \geq i$ not all
  equal. The correspondence $\gamma_j \circ {}^t\Gamma_j \circ L^{i-d}
  \circ \Gamma_{j'} \circ \gamma_{j'}$ is a cycle of dimension
  $j+j'-2i$ in the Chow group of $P_{ \lfloor j' \rfloor} \times P_{
    \lfloor j\rfloor}$, $C_{ \lfloor j'\rfloor} \times P_{ \lfloor
    j\rfloor}$, $P_{ \lfloor j'\rfloor} \times C_{ \lfloor j\rfloor}$
  or $C_{ \lfloor j'\rfloor} \times C_{ \lfloor j\rfloor}$ depending
  on the parity of $j$ and $j'$. Notice that $j+j'-2i \geq 1$, and that
  $j+j'-2i =1$ implies that $j'=i+1$ and $j=i$, and that $j+j'-2i =2$
  implies that $j$ and $j'$ have same parity (in fact $j=j'=i+1$ or
  $j' = j+2 = i+2$). The proof that $\gamma_j \circ {}^t\Gamma_j \circ
  L^{i-d} \circ \Gamma_{j'} \circ \gamma_{j'} = 0$ in each of these
  cases is then similar to the cases treated in the proof of the
  previous proposition.
\end{proof}

\paragraph{Proof of theorem \ref{projector3}.} We proceed by
induction on $k \geq 0$ to prove property

\noindent $\mathcal{P}_k$ : There exist idempotents $\Pi_i \in CH_d(X
\times X)$ for $0 \leq i \leq 2d$ such that
\begin{itemize}
\item $\Pi_j \circ \Pi_i = 0$ if $j-i < k$ and $j \neq i$.
\item $\Pi_{2i}$ satisfies the properties listed in theorem
  \ref{projector1} for all $i$.
\item $\Pi_{2i+1}$ satisfies the properties listed in theorem
  \ref{projector2} for all $i$.
\item The $\Pi_i$'s satisfy the conclusion of lemma \ref{leflemma}.
\end{itemize}

Clearly if property $\mathcal{P}_{2d+1}$ holds, then the idempotents
$\Pi_i$ are mutually orthogonal. (Actually it is enough to settle
$\mathcal{P}_3$ by remark \ref{nottrue}).  If we set $\Pi_i :=
\Pi_{i,\lfloor i/2 \rfloor}$, we see thanks to theorems
\ref{projector1} and \ref{projector2}, remark \ref{remark} and lemma
\ref{leflemma} that property $\mathcal{P}_0$ holds. Let's suppose that
property $\mathcal{P}_k$ holds and let's prove that
$\mathcal{P}_{k+1}$ holds. \medskip

We set $$P_i := (1-\frac{1}{2}\Pi_{2d}) \circ
(1-\frac{1}{2}\Pi_{2d-1}) \circ \cdots \circ (1-\frac{1}{2}\Pi_{i+1})
\circ \Pi_{i} \circ (1-\frac{1}{2}\Pi_{i-1}) \circ \cdots \circ
(1-\frac{1}{2}\Pi_{0}).$$ By lemma \ref{linalg}, these define
idempotents such that $P_j \circ P_i = 0$ if $j-i < k+1$ and $j \neq
i$.. It remains to check that $P_i$ enjoys the same properties as
$\Pi_{i}$.

It is straightforward from the formula that we have ${}^tP_i =
P_{2d-i}$. It is also straightforward that $P_i$ induces the projector
$H_*(X) \r N^{\lfloor i/2 \rfloor} H_i(X) \r H_*(X)$ in homology.
\medskip

Let's now consider an integer $i \geq d$ and prove that the Lefschetz
correspondence $L^{i-d}$ induces an isomorphism of Chow motives
$(X,P_i,0) \r (X,P_{2d-i},i-d)$. In fact, we are going to show that
${}^tP_i \circ L^{i-d} \circ P_i$ admits $P_i \circ \Gamma_i \circ
\gamma_i \circ {}^t\Gamma_i \circ {}^tP_i$ as an inverse, i.e. that
  $$(P_i \circ \Gamma_i \circ \gamma_i \circ {}^t\Gamma_i \circ
  {}^tP_i) \circ ({}^tP_i \circ L^{i-d} \circ P_i) = P_i$$ and
  $$({}^tP_i \circ L^{i-d} \circ P_i) \circ  (P_i \circ \Gamma_i \circ
  \gamma_i \circ {}^t\Gamma_i \circ {}^tP_i) = {}^tP_i.$$ Because $L =
  {}^tL$ and $\gamma_i = {}^t\gamma_i$, the second equality is the
  transpose of the first one. Therefore it is enough to establish the
  first equality. Thanks to remark \ref{remark} we have $\Pi_j \circ
  \Gamma_i \circ \gamma_i= 0$ for all $j <i$ and by transposing
  $\gamma_i \circ {}^t\Gamma_i \circ {}^t\Pi_j = 0$ for all $j<i$.
  Expanding $P_i$, we therefore see that $$P_i \circ \Gamma_i \circ
  \gamma_i \circ {}^t\Gamma_i \circ {}^tP_i = (1-\frac{1}{2}\Pi_{2d})
  \circ \cdots \circ (1-\frac{1}{2}\Pi_{i+1}) \circ \Pi_{i} \circ
  \Gamma_i \circ \gamma_i \circ {}^t\Gamma_i \circ {}^t\Pi_i \circ
  (1-\frac{1}{2}{}^t\Pi_{i+1}) \circ \cdots \circ
  (1-\frac{1}{2}{}^t\Pi_{2d}).$$ On the other hand, lemma
  \ref{leflemma} implies that $${}^tP_i \circ L^{i-d} \circ P_i =
  (1-\frac{1}{2}{}^t\Pi_{0}) \circ \cdots \circ
  (1-\frac{1}{2}{}^t\Pi_{i+1}) \circ {}^t\Pi_i \circ L^{i-d} \circ
  \Pi_i \circ (1-\frac{1}{2}\Pi_{i-1}) \circ \cdots \circ
  (1-\frac{1}{2}\Pi_{0}).$$ Put altogether, this gives
  $$(P_i \circ
  \Gamma_i \circ \gamma_i \circ {}^t\Gamma_i \circ {}^tP_i) \circ
  ({}^tP_i \circ L^{i-d} \circ P_i) = $$ $$(1-\frac{1}{2}\Pi_{2d})
  \circ \cdots \circ (1-\frac{1}{2}\Pi_{i+1}) \circ \Pi_{i} \circ
  \Gamma_i \circ \gamma_i \circ {}^t\Gamma_i \circ {}^t\Pi_i \circ
  L^{i-d} \circ \Pi_i \circ (1-\frac{1}{2}\Pi_{i-1}) \circ \cdots
  \circ (1-\frac{1}{2}\Pi_{0}).$$ By proposition \ref{lef1} if $i$ is
  even and by proposition \ref{lef2} if $i$ is odd, we have $\Pi_{i}
  \circ \Gamma_i \circ \gamma_i \circ {}^t\Gamma_i \circ {}^t\Pi_i
  \circ L^{i-d} \circ \Pi_i = \Pi_i$. This finishes the proof of the
  Lefschetz isomorphism. \medskip

  Let's now prove that the $P_i$'s satisfy the conclusion of lemma
  \ref{leflemma}. A careful look at the proof of lemma \ref{leflemma}
  shows that it is enough to show that $P_j$ factors through $\Gamma_j
  \circ {}^t\gamma_j$ if $\Pi_j$ does. This can be read immediately
  from the formula defining $\Pi_j$. \medskip

  If the projectors $\Pi_{2i}$ factor through a $0$-dimensional
  variety and if the projectors $\Pi_{2i+1}$ factor through a curve
  for all $i$, then it is clear from the formula that so will the
  projectors $P_{2i}$ and $P_{2i+1}$. On the one hand, proposition
  \ref{hom} gives $CH_i(X)_\hom = \ker\big( P_{2i} : CH_i(X) \r
  CH_i(X)\big)$ and proposition \ref{jacobi} gives $\ker \big( AJ_i :
  CH_i(X)_{\alg} \r J_i^a(X) \otimes \Q \big) = \ker \big(P_{2i+1} :
  CH_i(X)_{\alg} \r CH_i(X)_{\alg} \big).$ On the other hand,
  proposition \ref{Chow1} shows that $(X,P_{2i},0)$ is isomorphic to
  $(\mathds{L}^{\otimes i})^{\oplus d_i}$ and proposition \ref{Chow2}
  shows that $(X,P_{2i+1},0)$ is isomorphic to
  $\h_1(J_i^a(X))(i)$. Alternately, the conclusion of theorem \ref{GS}
  gives these isomorphisms of Chow motives.
\qed

\section{Representability of Chow groups and finite dimensional
  motives}

Given a smooth projective complex variety $X$ of dimension $d$, its
$i^\mathrm{th}$ Deligne cohomology group $H_i^\mathcal{D}(X,\Z(p))$ is
the $(2d-i)^\mathrm{th}$ hypercohomology group of the complex
$\Z_\mathcal{D}(d-p)$ given by $ 0 \r \Z \stackrel{\cdot
  (2i\pi)^{d-p}}{\longrightarrow} \mathcal{O}_X \r \Omega_X \r \cdots
\r \Omega_X^{d-p-1} \r 0$. In other words,
 $$H_i^\mathcal{D}(X,\Z(p)) =
 \mathds{H}^{2d-i}(X,\Z_\mathcal{D}(d-p)).$$ Deligne cohomology comes
 with a cycle map $cl_i^\mathcal{D} : CH^\Z_i(X) \r
 H_{2i}^\mathcal{D}(X,\Z(i))$ defined on the integral Chow group
 $CH^\Z_i(X)$ which is functorial with respect to the action of
 correspondences and fits into an exact sequence $$ 0 \r J_i(X) \r
 H_{2i}^\mathcal{D}(X,\Z(i)) \r \mathrm{Hdg}_{2i}(X) \r 0$$ where
 $\mathrm{Hdg}_{2i}(X)$ denotes the Hodge classes in $H_{2i}(X,\Z)$
 and $J_i(X)$ is Griffiths' $i^{th}$ intermediate Jacobian.  As proved
 in \cite[Prop. 1]{ZZ}, the following diagram with exact rows commutes
 \vspace{-25pt}
 \begin{center} \begin{equation}\label{eq} \xymatrix{0 \ar[r] &
       CH^\Z_i(X)_{\hom} \ar[d]^{AJ_i} \ar[r] & CH^\Z_i(X)
       \ar[d]^{cl_i^\mathcal{D}}
       \ar[r] & CH^\Z_i(X)/\hom \ar[d]^{cl_i} \ar[r] & 0  \\
       0 \ar[r] & J_i(X) \ar[r] & H_{2i}^\mathcal{D}(X) \ar[r] &
       \mathrm{Hdg}_{2i}(X) \ar[r] & 0. } \end{equation} \end{center}
 The homomorphism $cl_i : CH^\Z_i(X)/\hom \r \mathrm{Hdg}_{2i}(X)$ is
 always injective by definition of homological equivalence. In
 particular the functoriality of the Deligne cycle class map implies
 the functoriality of the Abel-Jacobi map with respect to the action
 of correspondences.

\begin{lemma} \label{essential} Let $i$ be an integer such that $d
  \leq 2i \leq 2d$.
  \begin{itemize}
  \item If the map $cl : CH_i(X) \r H_{2i}(X)$ is surjective then
    $H_{2i}(X) = N^iH_{2i}(X)$ and $H_{2d-2i}(X) =
    N^{d-i}H_{2d-2i}(X)$.
  \item If the map $cl^{\mathcal{D}} : CH_i(X) \r
    H^{\mathcal{D}}_{2i}(X)$ is surjective then $H_{2i+1}(X)= {{N}}^i
    H_{2i+1}(X)$ and $H_{2d-2i-1}(X) = N^{d-i-1}H_{2d-2i-1}(X)$.
\end{itemize}
\end{lemma}
\begin{proof}
  If the map $cl : CH_i(X) \r H_{2i}(X)$ is surjective then by
  definition $H_{2i}(X) = N^iH_{2i}(X)$. Because the Lefchetz
  isomorphism $L^{2i-d} : H_{2i}(X) \r H_{2d-2i}(X)$ is induced by a
  correspondence we also see that $H_{2d-2i}(X) =
  N^{d-i}H_{2d-2i}(X)$.

  Now suppose that the map $cl^{\mathcal{D}} : CH_i(X) \r
  H^{\mathcal{D}}_{2i}(X)$ is surjective.  A simple diagram chase in
  diagram \ref{eq} shows that the Abel-Jacobi map $AJ_i : CH_i(X)_\hom
  \r J_i(X)\otimes \Q$ is then surjective.  The Griffiths group
  $\mathrm{Griff}_i(X)$ being countable, this is possible only if
  $J_i(X) \otimes \Q = J_i^\alg(X) \otimes \Q$.  Therefore we have
  $J_i(X)\otimes \Q = J_i^a(X) \otimes \Q$ and hence $H_{2i+1}(X)=
  {{N}}^i H_{2i+1}(X)$.  Again because the Lefchetz isomorphism
  $L^{2i-d+1} : H_{2i+1}(X) \r H_{2d-2i-1}(X)$ is induced by a
  correspondence we also see that $H_{2d-2i-1}(X) =
  N^{d-i-1}H_{2d-2i-1}(X)$.
 \end{proof}

\subsection{From finite dimensionality to representability : proof of
  $2 \Rightarrow 1$ in theorem \ref{theorem2}}

First we need a standard
lemma.

\begin{lemma} \label{2} Let $N$ be a finite dimensional Chow motive.
  If its homology groups $H_*(N)$ vanish then $N=0$.
\end{lemma}
\begin{proof}
  The homology class of $\id_N \in \End_{k}(N)$ is then $0$. Kimura
  \cite[Prop. 7.2]{Kimura} proved that if a Chow motive $N$ is finite
  dimensional then the ideal of correspondences in $\End_{k}(N)$ which
  are homologically trivial is a nilpotent ideal. Hence $\id_N$ is
  nilpotent i.e. $\id_N =0$.
\end{proof}

\noindent \textit{Proof of $2 \Rightarrow 1$.} Lemma \ref{essential}
shows that $H_i(X) = N^{\lfloor i/2 \rfloor}H_i(X)$ for all $i$.
Therefore by lemma \ref{pairing-proj} the pairings $N^{\lfloor i/2
  \rfloor}H_i(X) \times N^{\lfloor (2d-i)/2 \rfloor}H_{2d-i}(X) \r \Q$
are all non degenerate. Theorems \ref{projector1}, \ref{projector2}
and \ref{projector3} then show that $A:= \mathds{1} \oplus
\h_1(\mathrm{Alb}_X) \oplus \L^{\oplus b_{2}} \oplus \h_1(J^a_1(X))(1)
\oplus (\L^{\otimes 2})^{\oplus b_{4}} \oplus \cdots \oplus
\h_1(J^a_{d-1}(X))(d-1) \oplus \L^{\otimes d}$ is a direct summand of
the Chow motive $\h(X)$ and that $H_*(A) = H_*(X)$.  Let's write
$\h(X) = A \oplus N$. The property of being finite dimensional is
stable by direct summand. Therefore $N$ is a finite dimensional
motive. Moreover $H_*(N)=0$. Lemma \ref{2} shows that $N=0$. \qed

\subsection{Representability vs. injectivity of the Abel-Jacobi maps :
  proof of $3 \Leftrightarrow 4$ in theorem \ref{theorem2}}

The results in this section are seemingly well-known. Given a smooth
projective complex variety $X$, we prove that the
following statements are equivalent :
\begin{enumerate}
\item $CH_i(X)_\alg$ is representable for all $i$.
\item The total Abel-Jacobi map $\bigoplus_i CH_i(X)_{\hom} \r
  \bigoplus_i J_i(X) \otimes \Q$ is injective.
\item The total Deligne cycle class map $cl_{\mathcal{D}} :
  \bigoplus_i CH_i(X)_{\hom} \r \bigoplus_i
  H_{2i}^{\mathcal{D}}(X,\Q(i))$ is injective.
\item The total Deligne cycle class map is bijective and
  $H_i(X)=N^{\lfloor i/2\rfloor }H_i(X)$ for all $i$.
\end{enumerate}

The equivalence ($2 \Leftrightarrow 3$) follows immediately from
diagram \ref{eq}.  The implication ($4 \Rightarrow 3$) is obvious and
the implication ($3 \Rightarrow 4$) is due to Esnault and Levine
\cite{EL} (theorem \ref{ELdeligne} below together with lemma
\ref{essential}). The main argument is a generalized decomposition of
the diagonal as performed by Laterveer \cite{Laterveer} and Paranjape
\cite{Paranjape} among others after Bloch's and Srinivas' original
paper \cite{BS}. Proposition \ref{bloch} proves the standard
implication ($2 \Rightarrow 1$). We couldn't find any reference for
the implication ($1 \Rightarrow 2$) so we include a proof of it, see
corollary \ref{rep_inj}. The proof goes through a generalized
decomposition of the diagonal as done in \cite[Theorem 1.2]{EL} with
some minor changes (theorem \ref{decdiagonal}).

\begin{theorem} [Esnault-Levine] \label{ELdeligne} Let $s$ be an
  integer.  Assume that the rational Deligne cycle class maps
  $cl_i^\mathcal{D} : CH_i(X) \r H^\mathcal{D}_{2i}(X)$ are injective
  for all $i \leq s$. Then these are all surjective. Moreover the
  rational cycle class maps $cl_i : CH_i(X) \r H_{2i}(X)$ are also
  surjective for all $i \leq s$.
\end{theorem}

\begin{proof} The fact that the rational Deligne cycle class maps are
  surjective for all $i \leq s$ is contained in Theorem 2.5 of
  \cite{EL} (the maps $cl_i^\mathcal{D}$ are denoted $cl_{0,0}^{2d-i}$
  in \cite{EL}). The claim about the rational cycle class maps being
  surjective is Corollary 2.6 (which states that
  $N^iH_{2i}(X)=Hdg_{2i}(X) \otimes \Q$) together with Theorem 3.2
  (which states in particular that $H_{2i}(X)=Hdg_{2i}(X) \otimes \Q$)
  of \cite{EL}.
 \end{proof}

\begin{proposition} \label{bloch} Given $i$, if the Abel-Jacobi map
  $CH_i(X)_{\mathrm{alg}} \r J_i(X) \otimes \Q$ is
  injective, then $CH_i(X)_{\mathrm{alg}}$ is  representable.
\end{proposition}

\begin{proof}
  Let $J_i^a(X)$ be the image of the Abel-Jacobi map $AJ_i :
  CH^\Z_i(X)_{\mathrm{alg}} \r J_i(X)$. By definition of algebraic
  equivalence, $CH_i(X)_\alg := \sum \im \big(\Gamma_* : CH_0(C)_\hom
  \r CH_i(X) \big)$ where the sum runs through all smooth projective
  curves $C$ and all correspondences $\Gamma \in CH_{i+1}(C \times
  X)$. Therefore, by functoriality of the Abel-Jacobi map, we have
  $J_i^a(X) = \sum \im \big(\Gamma_* : J(C) \r J_i^a(X) \big)$. By
  finiteness properties of abelian varieties there exist a curve and a
  correspondence $\Gamma \in CH_{i+1}(C \times X)$ such that $J_i^a(X)
  = \Gamma_*J(C)$. Therefore for this particular curve
  $\Gamma_*CH_0(C)_\hom = CH_i(X)_\alg$.
\end{proof}

\begin{theorem} \label{decdiagonal} Let $s$ be an integer with $0 \leq
  s \leq d$ and let $X$ be a $d$-dimensional smooth projective complex
  variety.  Assume $CH_i(X)_{\alg}$ is representable for all $i \leq
  s$. Then, there is a decomposition $$\Delta_X = \gamma_0 + \gamma_1
  + \cdots + \gamma_s + \gamma^{s+1} \in CH_d(X \times X)$$ such that
  $\gamma_i$ is supported on $D^i \times \Gamma_{i+1}$ for some
  subschemes $D^i$ and $\Gamma_{i+1}$ of $X$ satisfying $\dim D^i =
  d-i$ and $\dim \Gamma_{i+1} = i+1$ and $\gamma^{s+1}$ is supported
  on $D^{s+1} \times X$ for some subscheme $D^{s+1}$ of $X$ satisfying
  $\dim D^{s+1} = d-s-1$.
\end{theorem}

\begin{proof}
  The proof is the same as the proof of \cite[lemma 1.1]{EL} once one
  remarks that the map $ch: CH_0(\tilde{D}) \r CH^n(X)$ on page 207
  has image contained in $CH^n(X)_{\alg}$ and therefore factors
  through the Albanese map $CH_0(\tilde{D}) \r
  \mathrm{Alb}_{\tilde{D}}$, because $CH^n(X)_{\alg}$ is
  representable and has thus the structure of an abelian variety.
\end{proof}

\begin{corollary} \label{rep_inj} Assume $CH_i(X)_{\alg}$ is
  representable for all $i \leq s$. Then the Abel-Jacobi maps $AJ_i :
  CH_i(X)_{\hom} \r J_i(X) \otimes \Q$ are injective for all $i \leq
  s$.
\end{corollary}
\begin{proof} By assumption made on $CH_*(X)_{\alg}$, the diagonal
  $\Delta_X$ admits a decomposition as in theorem \ref{decdiagonal}.
  For all $i \leq s$, let $\widetilde{\Gamma}_{i+1}$ be a
  desingularization of $ \Gamma_{i+1}$. The action of the
  correspondence $\gamma_i$ on $CH_j(X)_\hom$ then factors through
  $CH_j(\widetilde{\Gamma}_{i+1})_{\hom}$.  For dimension reasons
  $\gamma_i$ acts possibly non trivially only on $CH_i(X)$ and
  $CH_{i+1}(X)$. Also for dimension reasons, the correspondence
  $\gamma^{s+1}$ acts trivially on $CH_i(X)$ for $i \leq s$. Therefore
  the cycle $\gamma_i$ acts non trivially on $CH_j(X)_{\hom}$ only if
  $i=j$. Thus the action of $\gamma_i$ on $CH_i(X)_{\hom}$ is
  identity.  Finally, by functoriality of the algebraic Abel-Jacobi
  map, we have the following commutative diagram for all $i \leq s$
  \begin{center} $ \xymatrix{ CH_i(X)_{\hom} \ar[d]^{AJ_i} \ar[r] &
      CH_i(\widetilde{\Gamma}_{i+1})_{\hom} \ar[d]^{\simeq}
      \ar[r] &  CH_i(X)_{\hom} \ar[d]^{AJ_i} \\
      J_i(X) \ar[r] & J_i(\widetilde{\Gamma}_{i+1}) \ar[r] & J_i(X).
    }$ \end{center} The composition of the two maps on each row is
  induced by $\gamma_i$ and is equal to identity up to torsion. A
  diagram chase then shows that $AJ_i : CH_i(X)_{\alg} \r
  J_i^a(X)\otimes \Q$ is injective.
\end{proof}

\begin{remark}
  Given $i$, I cannot prove that if $CH_i(X)_\alg$ is representable
  then the Abel-Jacobi map $AJ_i : CH_i(X)_{\alg} \r J_i(X) \otimes
  \Q$ restricted to algebraically trivial cycles is injective.
\end{remark}
\begin{remark} \label{gendiag} Bloch and Srinivas proved \cite[Theorem
  1(i)]{BS} that if $CH_0(X)_{\alg}$ is representable then so is
  $CH^2(X)_{\alg}$. A generalized decomposition of the diagonal shows
  that if $CH_0(X)_{\alg}, \ldots, CH_s(X)_{\alg}$ are representable
  then $CH^2(X)_{\alg}, \ldots, CH^{2+s}(X)_{\alg}$ are also
  representable. Therefore, if $d$ is the dimension of $X$, it is
  enough to know that $CH_0(X)_{\alg}, \ldots, CH_{\lfloor d/2\rfloor
    -1}(X)_{\alg}$ are representable in order to deduce that
  $CH_*(X)_{\alg}$ is representable.
\end{remark}

\subsection{From representability to finite dimensionality : proof of
  $4 \Rightarrow 1$ in theorem \ref{theorem2}}

In order to prove the implication $4 \Rightarrow 1$ of theorem
\ref{theorem2}, we again use our projectors $\Pi_{2i,i}$ and
$\Pi_{2i+1,i}$, together with the following lemma which appears in
\cite[lemma 1]{GG}.

\begin{lemma} \label{1} Let $N$ be a Chow motive over a field $k$ and
  let $\Omega$ be a universal domain over $k$, i.e. an algebraically
  closed field of infinite transcendence degree over $k$. If
  $CH_*(N_\Omega)=0$, then $N=0$. \qed
\end{lemma}

\noindent \textit{Proof of $4 \Rightarrow 1$.}  If $CH_*(X_\C)_\alg$
is representable then corollary \ref{rep_inj} shows that the Deligne
cycle class maps $cl_i^{\mathcal{D}}$ are all injective. By Esnault
and Levine's theorem \ref{ELdeligne}, the Deligne cycle class maps
$cl_i^{\mathcal{D}}$ and the cycle class maps $cl_i$ are surjective
for all $i$.  Now lemma \ref{essential} shows that $H_{2i}(X) =
N^iH_{2i}(X)$ and $H_{2i+1}(X) = N^iH_{2i+1}(X)$ for all $i$.  Thanks
to lemma \ref{pairing-proj} we can therefore apply theorems
\ref{projector1}, \ref{projector2} and \ref{projector3} to cut out the
motive $\mathds{1} \oplus \h_1(\mathrm{Alb}_X) \oplus \L^{\oplus
  b_{2}} \oplus \h_1(J^a_1(X))(1) \oplus (\L^{\otimes 2})^{\oplus
  b_{4}} \oplus \cdots \oplus \h_1(J^a_{d-1}(X))(d-1) \oplus
\L^{\otimes d}$ from $\h(X)$. These two motives have same rational
Chow groups when the base field is extended to $\C$, lemma \ref{1}
implies they are equal. \qed \medskip

As a corollary, we obtain a result proved independently by Kimura
\cite{Kimura2} (Kimura's result works more generally for any pure Chow
motive over $\C$).

\begin{proposition} Let $X$ be a $d$-dimensional smooth projective
  variety over $k$. If $CH_*(X_\C)$ is a finite dimensional
  $\Q$-vector space, then $$\h(X) \simeq \bigoplus_{i=0}^d
  (\L^{\otimes i})^{\oplus b_{2i}}.$$ Moreover, the cycle class maps
  $cl_i : CH_i(X) \r H_{2i}(X)$ are all isomorphisms
 \end{proposition}

 \begin{proof} Indeed, if $CH_*(X_\C)$ is a finite dimensional
   $\Q$-vector space then it is representable. Apply theorem
   \ref{theorem2} to see that $\h(X)$ is a direct sum of Lefschetz
   motives and twisted $\h_1$'s of abelian varieties. Now for a
   complex abelian variety $J$, $CH_0(\h_1(J)) = J \otimes \Q$ which
   is an infinite dimensional $\Q$-vector space if $J \neq 0$.
   Therefore $\h(X)$ is a direct sum of Lefschetz motives only.
 \end{proof}

\section{Chow-K\"unneth decompositions} \label{example}

A smooth projective variety $X$ of dimension $d$ is said to have a
Chow-K\"unneth decomposition (CK decomposition for short) if there
exist mutually orthogonal idempotents $\Pi_0, \Pi_1, \ldots \Pi_{2d}
\in CH_d(X \times X)$ adding to the identity $\Delta_X$ such that
$(\Pi_i)_*H_*(X)=H_i(X)$ for all $i$.  In this section, we wish to
give explicit examples of varieties having a Chow-K\"unneth
decomposition. For this purpose we use the projectors of theorems
\ref{projector1} and \ref{projector2}.  Along the way we are able to
establish Grothendieck's standard conjectures and Kimura's finite
dimensionality conjecture in some new cases. In \cite{Vial2}, we prove
Murre's conjectures for all the varieties considered in \S\S4.2 and
4.3 (and some others as well).

\subsection{The basic theorem}









Here, $X$ denotes a smooth projective variety of dimension $d$. All
the varieties for which we will be able to show that they admit a CK
decomposition will also satisfy the motivic Lefschetz conjecture.

\begin{definition}
  The variety $X$ is said to satisfy \textit{the motivic Lefschetz
    conjecture} if it admits a CK decomposition $\{\Pi_i\}_{0 \leq i
    \leq 2d}$ such that, for all $i > d$ the morphism of Chow
  motives $(X,\Pi_i,0) \r (X,\Pi_{2d-i}, i-d)$ given by intersecting
  $i-d$ times with a hyperplane section is an isomorphism.
\end{definition}

It is immediate to see that if $X$ satisfies the motivic Lefschetz
conjecture then it satisfies the Lefschetz standard conjecture. Since
Grothendieck's standard conjectures for $X$ reduce to the standard
Lefschetz conjecture for $X$ in characteristic zero
\cite[5.4.2.2]{An}, we get that $X$ satisfies all of Grothendieck's
standard conjectures.

The key result of this section is the following.

\begin{theorem} \label{basictheorem} If $H_i(X) = N^{\lfloor i/2
    \rfloor}H_i(X)$ for all $i>d$, then $X$ has a Chow-K\"unneth
  decomposition $\{P_i\}_{0 \leq i \leq 2d}$ where the idempotents
  $P_i$ for $i \neq d$ satisfy all the properties listed in theorems
  \ref{projector1} and \ref{projector2}. In particular the motivic
  Lefschetz conjecture holds for $X$ and hence, also do the standard
  conjectures hold for $X$.
\end{theorem}

\begin{proof} If $H_i(X) = N^{\lfloor i/2 \rfloor}H_i(X)$ for all
  $i>d$, then by Poincar\'e duality the pairings $N^{\lfloor
    \frac{i}{2} \rfloor}H_i(X) \times N^{\lfloor \frac{2d-i}{2}
    \rfloor}H_{2d-i}(X) \r \Q$ are non degenerate for all $i>d$. The
  motivic Lefschetz isomorphisms of theorems \ref{projector1} and
  \ref{projector2} then imply that $X$ satisfies the Lefschetz
  standard conjecture. Therefore, by lemma \ref{pairing-proj}, the
  pairing $N^{\lfloor d/2 \rfloor}H_d(X) \times N^{\lfloor d/2
    \rfloor}H_d(X) \r \Q$ is also non degenerate. By theorem
  \ref{projector3}, we get mutually orthogonal idempotents
  $\{\Pi_i\}_{0 \leq i \leq 2d}$. Let's put $P_i := \Pi_i$ for $i \neq
  d$ and $P_d := \Delta_X - \sum_{i\neq d}\Pi_i$. Then $\{P_i\}_{0
    \leq i \leq 2d}$ is the required CK decomposition for $X$.
\end{proof}

\subsection{Some examples of varieties having a CK decomposition and
  satisfying the motivic Lefschetz conjecture}

An immediate consequence to theorem \ref{basictheorem} is the
following.

\begin{corollary} \label{C} Let $Y$ be a $3$-fold with $H^2(Y,O_Y)=0$,
  e.g. a Calabi-Yau $3$-fold. Then $Y$ has a CK decomposition and the
  motivic Lefschetz conjecture holds for $Y$.
\end{corollary}
\begin{proof}
  By the Lefschetz $(1,1)$-theorem, $H^2(Y,O_Y)=0$ implies
  $H_4(Y)=N^2H_4(Y)$. Thus $H_i(X) = N^{\lfloor i/2 \rfloor}H_i(X)$ for all
  $i>3$.
\end{proof}


In theorems \ref{B} and \ref{C} below, in addition to proving that
some $X$ has a CK decomposition, we give some information on the
support of the middle CK projector of $X$. Such information will be
used in \cite{Vial2} to further prove Murre's conjectures in the cases
covered by the theorems. For this purpose we need the following lemma
that appears in \cite[lemma 1.2]{Vial1} and which is essentially due
to Kahn and Sujatha \cite{KahnSujatha}.

\begin{lemma} \label{decbir} Let $M = (X,p,0)$ be a Chow motive over
  $k$ and let $\Omega$ be a universal domain over $k$. If
  $CH_0(N_\Omega)=0$, then there exists a Chow motive $N=(Y,q,0)$ such
  that $M=(Y,q,1)$. \qed
\end{lemma}

\begin{theorem} \label{B} Let $X$ be a smooth projective variety of
  even dimension $d=2n$. If $CH_0(X)_\alg, CH_1(X)_\alg, \ldots,
  CH_{n-2}(X)_\alg$ are representable, then $X$ has a CK decomposition
  $\{\Pi_i\}$ and the motivic Lefschetz conjecture holds for $X$.
  Moreover, the idempotents $\Pi_i$ are as in theorems
  \ref{projector1} and \ref{projector2} for $i \neq d$ and the
  idempotent $\Pi_d$ has a representative supported on $X \times Z$
  with $Z$ a subvariety of $X$ of dimension $n+1$.
\end{theorem}

\begin{proof}
  By a generalized decomposition of the diagonal (as performed for
  instance by Laterveer \cite[2.1]{Laterveer}), the assumption on the
  Chow groups of $X$ implies that $H_i(X) = N^{\lfloor i/2
    \rfloor}H_i(X)$ for all $i>d$. We can therefore apply theorem
  \ref{basictheorem} to get a CK decomposition $\{\Pi_i\}_{0 \leq i
    \leq 2d}$ for $X$ where the idempotents $\Pi_i$ for $i \neq d$
  satisfy all the properties listed in theorems \ref{projector1} and
  \ref{projector2}. By corollary \ref{rep_inj} if $CH_0(X)_\alg,
  CH_1(X)_\alg, \ldots, CH_{n-2}(X)_\alg$ are representable, then the
  Abel-Jacobi maps $AJ_i : CH_i(X)_\hom \r J_i(X) \otimes \Q$ are
  injective for all $i \leq n-2$. Esnault and Levine's theorem
  \ref{ELdeligne} then implies that the Abel-Jacobi maps are
  bijective. Thanks to the properties of the CK projectors, we thus
  get $CH_i(X) = (\Pi_{2i} + \Pi_{2i+1})_*CH_i(X)$ for all $i \leq
  n-2$. As such, the idempotent $\Pi_d$ acts trivially on $CH_i(X)$
  for all $i \leq n-2$. By applying $n-1$ times lemma \ref{decbir}, we
  get that $(X,\Pi_d,0)$ is isomorphic to some Chow motive
  $(Y,q,n-1)$. This means that there exists a correspondence $f \in
  \Hom((X,\Pi_d,0) ,(Y,q,n-1))$ such that $\Pi_d = \Pi_d \circ f^{-1}
  \circ q \circ f \circ \Pi_d$. In particular $\Pi_d$ factors through
  $Y$ and a straightforward analysis of the dimensions shows that
  $\Pi_d$ has a representative supported on $X \times Z$ with $Z$ a
  subvariety of $X$ of dimension $n+1$.
\end{proof}

\begin{corollary} \label{fourfold} Every fourfold $X$ with
  $CH_0(X)_\alg$ representable has a CK decomposition and satisfies
  the motivic Lefschetz conjecture and hence the standard conjectures.
\end{corollary}





\begin{corollary} Let $X$ be a smooth projective fourfold admitting a
  curve $C$ as a base for its maximal rationally connected fibration.
  This means that there exists a rational map $f : X \dashrightarrow
  C$ with rationally connected general fiber. Then $X$ has a CK
  decomposition and satisfies the motivic Lefschetz conjecture and
  hence the standard conjectures.
\end{corollary}
\begin{proof}
  $CH_0(X)_\alg$ is representable.
\end{proof}

\begin{corollary} \label{Fano} Every rationally connected smooth
  projective fourfold (e.g. every smooth projective variety which is
  birational to a Fano fourfold) has a Chow-K\"unneth decomposition
  and satisfies the motivic Lefschetz conjecture.
\end{corollary}
\begin{proof} A rationally connected smooth projective fourfold $X$
  satisfies $CH_0(X)=\Q$.
\end{proof}

\begin{remark}
  Arapura \cite{Arapura} proved the Lefschetz standard conjecture for
  unirational fourfolds. He does so by proving that a unirational
  fourfold is \textit{motivated} by surfaces.  More generally, Arapura
  proves that any variety which is motivated by a surface (this means
  that the cohomology of $X$ is generated by the cohomology of product
  of surfaces via correspondences) satisfies the standard Lefschetz
  conjecture.

  Corollary \ref{fourfold} is more precise for unirational fourfolds
  because we obtain the Lefschetz isomorphism modulo rational
  equivalence (rather than just modulo homological equivalence).
  Moreover, corollary \ref{fourfold} includes the case of rationally
  connected fourfolds as well as the case of fourfolds admitting a
  curve as a base for their maximal rationally connected fibration.

  Let's also mention that in what follows Arapura's technique doesn't
  seem to apply to prove the Lefschetz standard conjecture because the
  middle cohomology of the varieties in question is not necessarily
  generated by the cohomology of products of surfaces.
\end{remark}

\begin{theorem} \label{C} Let $X$ be a smooth projective variety of
  odd dimension $d=2n+1$ with $H^{n+1}(X,\Omega_X^{n-1})=0$. If
  $CH_0(X)_\alg, CH_1(X)_\alg, \ldots, CH_{n-2}(X)_\alg$ are
  representable, then $X$ has a CK decomposition $\{\Pi_i\}$ and the
  motivic Lefschetz conjecture holds for $X$. Moreover, the
  idempotents $\Pi_i$ are as in theorems \ref{projector1} and
  \ref{projector2} for $i \neq d$ and the idempotent $\Pi_d$ has a
  representative supported on $X \times Z$ with $Z$ a subvariety of
  $X$ of dimension $n+2$.
\end{theorem}
\begin{proof} As for the proof of theorem \ref{B}, a generalized
  decomposition of the diagonal argument shows that the assumption on
  the Chow groups of $X$ implies that $H_{i}(X)=N^{\lfloor i/2
    \rfloor}H_{i}(X)$ for $i > d+1$ and that $H_{d+1}(X) =
  N^{\frac{d-1}{2}}H_{d+1}(X)$. This last equality means that there is
  a smooth projective variety $S$ of dimension $n+2$ and a map $f : S
  \r X$ such that $f_*H^2(S) = H_{d+1}(X)$. Because
  $H^{n+1}(X,\Omega_X^{n-1})=0$, we see that $H_{d+1}(X)$ is made of
  Hodge classes. By the Lefschetz $(1,1)$-theorem applied to $S$, we
  see that $H_{d+1}(X)$ is spanned by algebraic cycles, i.e. that
  $H_{d+1}(X) = N^{\frac{d+1}{2}}H_{d+1}(X)$. We can thus apply
  theorem \ref{basictheorem} to get a CK decomposition $\{\Pi_i\}_{0
    \leq i \leq 2d}$ that satisfies the motivic Lefschetz theorem.
  The proof of the fact that $\Pi_d$ has a representative supported on
  $X \times Z$ for $Z$ a subvariety of $X$ of dimension $n+2$ goes
  along the same lines as the proof of theorem \ref{B}.
\end{proof}

\begin{corollary} Let $X$ be a smooth projective fivefold. If
  $CH_0(X)_\alg$ is representable and if $H^{3}(X,\Omega_X)=0$, then
  $X$ has a CK decomposition and satisfies the standard conjectures.
\end{corollary}

\begin{corollary} Let $X$ be a smooth projective rationally connected
  fivefold, e.g a fivefold which is birational to a Fano fivefold. If
  $H^{3}(X,\Omega_X)=0$, then $X$ has a CK decomposition and satisfies
  the standard conjectures.
\end{corollary}


 \subsection{Hypersurfaces of very small degree are Kimura finite dimensional}

 Otwinowska \cite{Otwinowska} proved that if $X$ is a smooth
 hyperplane section of a hypersurface in $\P^{n+1}$ covered by
 $l$-planes then $CH_i(X)_\hom = 0$ for $i \leq l-1$ (see also
 Esnault, Levine and Viehweg \cite{ELV}).  Therefore when $l=\lfloor
 n/2\rfloor$ the Chow groups $CH_i(X)_\alg$ are all representable by
 remark \ref{gendiag}. As a direct application of theorem
 \ref{theorem2} we get

 \begin{theorem} Let $l=\lfloor n/2\rfloor$ and let $X$ be a smooth
   hyperplane section of a hypersurface in $\P^{n+1}$ covered by
   $l$-planes. Then,
  \begin{itemize}
  \item  if $n-1$ is even, $\h(X)=\mathds{1} \oplus \L \oplus
    \L^{\otimes 2} \oplus \cdots \oplus  \L^{\otimes n-1}$.
  \item if $n-1$ is odd, $\h(X)=\mathds{1} \oplus \L \oplus \cdots
    \oplus \L^{\otimes l} \oplus \h_1(J_l^\alg)(l) \oplus \L^{\otimes
      l+1}\oplus \cdots \oplus \L^{\otimes n-1}$.
  \end{itemize}
  Moreover, in any case, $\h(X)$ is finite dimensional in the sense of
  Kimura.
\end{theorem}

\begin{remark}
  Otwinowska also mentions that if $k(n-l)- \binom{d+l}{d}+1 \geq 0$,
  any smooth projective hypersurface of degree $d$ in $\P_\C^n$ is
  covered by linear projective varieties of dimension $l$.
\end{remark}

\begin{examples}
  Here are some varieties for which theorem \ref{theorem2} and the
  results of \cite{ELV} make it possible to prove that they have
  finite dimensional Chow motive :
  \begin{itemize}
  \item Cubic $5$-folds.
  \item A $5$-fold which is the smooth intersection of a cubic and a
    quadric.
  \item A $7$-fold which is the smooth intersection of two quadrics.
  \end{itemize}
\end{examples}

Further examples of varieties with finite dimensional Chow motive can
be constructed as follows. Let $X$ be a variety as in the theorem
above. Consider smooth projective varieties obtained from $X$ by
successively blowing up smooth curves. Then, by the blowing-up formula
for Chow motives, such varieties have finite dimensional Chow motive.
Moreover any variety $Y$ which is dominated by a product of such
varieties has finite dimensional Chow motive.




\begin{footnotesize}
  \bibliographystyle{plain} 
  \bibliography{bib} \medskip

\def\cprime{$'$}
\begin{thebibliography}{10}

\bibitem{An}
Yves Andr{\'e}.
\newblock {\em Une introduction aux motifs (motifs purs, motifs mixtes,
  p\'eriodes)}, volume~17 of {\em Panoramas et Synth\`eses [Panoramas and
  Syntheses]}.
\newblock Soci\'et\'e Math\'ematique de France, Paris, 2004.

\bibitem{AnKa}
Yves Andr{\'e} and Bruno Kahn.
\newblock (with an appendix by {P}. {O}'{S}ullivan) {N}ilpotence, radicaux et
  structures mono\"\i dales.
\newblock {\em Rend. Sem. Mat. Univ. Padova}, 108:107--291, 2002.

\bibitem{Arapura}
Donu Arapura.
\newblock Motivation for {H}odge cycles.
\newblock {\em Adv. Math.}, 207(2):762--781, 2006.

\bibitem{BS}
S.~Bloch and V.~Srinivas.
\newblock Remarks on correspondences and algebraic cycles.
\newblock {\em Amer. J. Math.}, 105(5):1235--1253, 1983.

\bibitem{ZZ}
Fouad El~Zein and Steven Zucker.
\newblock Extendability of normal functions associated to algebraic cycles.
\newblock In {\em Topics in transcendental algebraic geometry ({P}rinceton,
  {N}.{J}., 1981/1982)}, volume 106 of {\em Ann. of Math. Stud.}, pages
  269--288. Princeton Univ. Press, Princeton, NJ, 1984.

\bibitem{EL}
H{\'e}l{\`e}ne Esnault and Marc Levine.
\newblock Surjectivity of cycle maps.
\newblock {\em Ast\'erisque}, (218):203--226, 1993.
\newblock Journ\'ees de G\'eom\'etrie Alg\'ebrique d'Orsay (Orsay, 1992).

\bibitem{ELV}
H{\'e}l{\`e}ne Esnault, Marc Levine, and Eckart Viehweg.
\newblock Chow groups of projective varieties of very small degree.
\newblock {\em Duke Math. J.}, 87(1):29--58, 1997.

\bibitem{GG}
S.~Gorchinskiy and V.~Guletskii.
\newblock {Motives and representability of algebraic cycles on threefolds over
  a field}.
\newblock {to appear in JAG}.

\bibitem{Griffiths}
Phillip~A. Griffiths.
\newblock On the periods of certain rational integrals. {I}, {II}.
\newblock {\em Ann. of Math. (2) 90 (1969), 460-495; ibid. (2)}, 90:496--541,
  1969.

\bibitem{Jannsen3}
Uwe Jannsen.
\newblock Motives, numerical equivalence, and semi-simplicity.
\newblock {\em Invent. Math.}, 107(3):447--452, 1992.

\bibitem{Jannsen}
Uwe Jannsen.
\newblock Motivic sheaves and filtrations on {C}how groups.
\newblock In {\em Motives (Seattle, WA, 1991)}, volume~55 of {\em Proc. Sympos.
  Pure Math.}, pages 245--302. Amer. Math. Soc., Providence, RI, 1994.

\bibitem{KahnSujatha}
B.~Kahn and R.~Sujatha.
\newblock {Birational motives, I pure birational motives}.
\newblock {Preprint, February 27, 2009.}

\bibitem{KMP}
Bruno Kahn, Jacob~P. Murre, and Claudio Pedrini.
\newblock On the transcendental part of the motive of a surface.
\newblock In {\em Algebraic cycles and motives. Vol. 2}, volume 344 of {\em
  London Math. Soc. Lecture Note Ser.}, pages 143--202. Cambridge Univ. Press,
  Cambridge, 2007.

\bibitem{Kimura}
Shun-Ichi Kimura.
\newblock Chow groups are finite dimensional, in some sense.
\newblock {\em Math. Ann.}, 331(1):173--201, 2005.

\bibitem{Kimura2}
Shun-ichi Kimura.
\newblock Surjectivity of the cycle map for {C}how motives.
\newblock In {\em Motives and algebraic cycles}, volume~56 of {\em Fields Inst.
  Commun.}, pages 157--165. Amer. Math. Soc., Providence, RI, 2009.

\bibitem{Laterveer}
Robert Laterveer.
\newblock Algebraic varieties with small {C}how groups.
\newblock {\em J. Math. Kyoto Univ.}, 38(4):673--694, 1998.

\bibitem{Murre}
J.~P. Murre.
\newblock On the motive of an algebraic surface.
\newblock {\em J. Reine Angew. Math.}, 409:190--204, 1990.

\bibitem{Otwinowska}
Anna Otwinowska.
\newblock Remarques sur les groupes de {C}how des hypersurfaces de petit
  degr\'e.
\newblock {\em C. R. Acad. Sci. Paris S\'er. I Math.}, 329(1):51--56, 1999.

\bibitem{Paranjape}
Kapil~H. Paranjape.
\newblock Cohomological and cycle-theoretic connectivity.
\newblock {\em Ann. of Math. (2)}, 139(3):641--660, 1994.

\bibitem{Scholl}
A.~J. Scholl.
\newblock Classical motives.
\newblock In {\em Motives (Seattle, WA, 1991)}, volume~55 of {\em Proc. Sympos.
  Pure Math.}, pages 163--187. Amer. Math. Soc., Providence, RI, 1994.

\bibitem{Vial2}
Charles Vial.
\newblock {Niveau and coniveau filtrations on cohomology groups and Chow
  groups}.
\newblock {Preprint}.

\bibitem{Vial1}
Charles Vial.
\newblock Pure motives with representable {C}how groups.
\newblock {\em C. R. Math. Acad. Sci. Paris}, 348(21-22):1191--1195, 2010.

\bibitem{Weil}
Andr{\'e} Weil.
\newblock {\em Vari\'et\'es ab\'eliennes et courbes alg\'ebriques}.
\newblock Actualit\'es Sci. Ind., no. 1064 = Publ. Inst. Math. Univ. Strasbourg
  8 (1946). Hermann \& Cie., Paris, 1948.

\end{thebibliography}

  \textsc{DPMMS, University of Cambridge, Wilberforce Road, Cambridge,
    CB3 0WB, United Kingdom}
\end{footnotesize}

\textit{e-mail :}  \texttt{C.Vial@dpmms.cam.ac.uk}

\end{document}